\documentclass[12pt, leqno]{amsart}

\usepackage{graphicx}
\usepackage{amssymb,amsmath,amsthm,amscd}
\usepackage{mathrsfs}
\usepackage{enumerate}
\usepackage[usenames,dvipsnames]{color}
\usepackage[colorlinks=true, pdfstartview=FitV,
 linkcolor=blue,citecolor=blue,urlcolor=blue]{hyperref}
\usepackage[all]{xy}

\usepackage{verbatim}

\numberwithin{equation}{section}
\allowdisplaybreaks[4]

\newcommand{\nc}{\newcommand}
\nc{\op}{\operatorname}

\theoremstyle{plain}
\newtheorem{lemma}{Lemma}[subsection]
\newtheorem{prop}[lemma]{Proposition}
\newtheorem{theorem}[lemma]{Theorem}
\newcommand{\Prop}{\begin{prop}}
\newcommand{\enprop}{\end{prop}}
\newcommand{\Lemma}{\begin{lemma}}
\newcommand{\enlemma}{\end{lemma}}
\newcommand{\Th}{\begin{theorem}}
\newcommand{\enth}{\end{theorem}}
\newtheorem{corollary}[lemma]{Corollary}
\newcommand{\Cor}{\begin{corollary}}
\newcommand{\encor}{\end{corollary}}
\newtheorem{definition}[lemma]{Definition}

\newcommand{\Def}{\begin{definition}}
\newcommand{\edf}{\end{definition}}
\newtheorem{sublemma}[lemma]{Sublemma}
\newcommand{\Sublemma}{\begin{sublemma}}
\newcommand{\ensub}{\end{sublemma}}

\theoremstyle{definition}
\newtheorem{remark}[lemma]{Remark}

\newtheorem{Convention}[lemma]{Convention}
\newcommand{\Conv}{\begin{Convention}}
\newcommand{\enconv}{\end{Convention}}
\newcommand{\Rem}{\begin{remark}}
\newcommand{\enrem}{\end{remark}}

\newcommand{\C}{\mathbb {C}}
\newcommand{\Q}{\mathbb {Q}}

\newcommand{\Z}{{\mathbb Z}}
\newcommand{\B}{{\mathbf{B}}}

\newcommand{\CC}{{\mathscr{C}}}
\newcommand{\one}{{\bf{1}}}
\newcommand{\seteq}{\mathbin{:=}}

\newcommand{\hd}{{\operatorname{hd}}}

\newcommand{\g}{{\mathfrak{g}}}

\newcommand{\Hom}{\operatorname{Hom}}

\newcommand{\isoto}[1][]{\mathop{\xrightarrow%
[{\raisebox{.3ex}[0ex][.3ex]{$\scriptstyle{#1}$}}]%
{{\raisebox{-.6ex}[0ex][-.6ex]{$\mspace{2mu}\sim\mspace{2mu}$}}}}}
\newcommand{\tensor}{\otimes}

\newcommand{\eq}{\begin{eqnarray}}
\newcommand{\eneq}{\end{eqnarray}}

\newcommand{\eqn}{\begin{eqnarray*}}
\newcommand{\eneqn}{\end{eqnarray*}}
\newcommand{\on}{\operatorname}

\newcommand{\bni}{\be[{\rm(i)}]}
\newcommand{\bna}{\be[{\rm(a)}]}

\newcommand{\QED}{\end{proof}}
\newcommand{\Proof}{\begin{proof}}

\newcommand{\soplus}{\mathop{\mbox{\normalsize$\bigoplus$}}\limits}

\newcommand{\To}[1][{\hs{2ex}}]{\xrightarrow{\,#1\,}}
\newcommand{\id}{\on{id}}
\newcommand{\ba}{\begin{array}}
\newcommand{\ea}{\end{array}}

\newcommand{\bi}{\begin{enumerate}[{\rm(i)}]}

\newcommand{\set}[2]{\left\{#1 \mathbin{;} #2 \right\}}
\newcommand{\supp}{\operatorname{supp}}
\newcommand{\Mod}{\operatorname{Mod}}
\newcommand{\Modg}{\operatorname{{Mod}_{\mathrm{gr}}}}
\newcommand{\hs}{\hspace*}

\newcommand{\eqsub}{\begin{subequations}\begin{eqnarray}}
\newcommand{\eneqsub}{\end{eqnarray}\end{subequations}}

\newcommand{\ol}{\overline}

\nc{\la}{\lambda}
\nc{\lam}{\lambda}
\nc{\U}[1][\g]{U_q(#1)}
\nc{\te}{\tilde{e}}
\nc{\tei}{\tilde{e}_i}
\nc{\tf}{\tilde{f}}
\nc{\tfi}{\tilde{f}_i}
\nc{\tU}{\widetilde U_q(\g)}
\nc{\tE}{\tilde{E}}
\nc{\tF}{\widetilde{\F}}

\nc{\tk}{\tilde{k}}
\nc{\tkone}{\tk_{\ol{1}}}
\nc{\teone}{\tilde{e}_{\ol{1}}}
\nc{\tfone}{\tilde{f}_{\ol{1}}}

\nc{\teibar}{\tilde{e}_{\ol{i}}} \nc{\tfibar}{\tilde{f}_{\ol{i}}}
\nc{\tki}{{\tk}_{\ol {i}}}

\nc{\BZ}{{\mathbb{Z}}}
\nc{\al}{\alpha}
\nc{\qs}{{q}}
\nc{\lan}{\langle}
\nc{\ran}{\rangle}
\nc{\re}{{\mathrm{re}}}
\nc{\wt}{\operatorname{wt}}
\nc{\ch}{\operatorname{ch}}
\nc{\Uf}[1][\g]{U^-_q(#1)}
\nc{\Ue}{U^+_q(\g)}
\nc{\eps}{\varepsilon}
\nc{\vphi}{\varphi}
\nc{\sphi}{\varphi^*}
\nc{\seps}{\varepsilon^*}

\nc{\nn}{\nonumber}

\nc{\vp}{\varpi}
\nc{\cls}{{\operatorname{cl}}}
\nc{\Wt}{{\operatorname{Wt}}}
\nc{\Us}{U'_q(\g)}
\nc{\La}{\Lambda}
\nc{\ro}{{\rm(}}
\nc{\rf}{{\rm)}}
\nc{\norm}{{\mathrm{norm}}}
\nc{\qbox}{\quad\mbox}
\nc{\braid}{{\mathfrak{B}}}
\nc{\Ad}{\operatorname{Ad}}
\nc{\Aut}{\operatorname{Aut}}
\nc{\dt}[1]{\tilde{\tilde #1}}
\nc{\Sn}{S^{{\mathrm{norm}}}}
\nc{\aff}{{\rm{aff}}}
\nc{\rk}{{\mathrm{rk}}}
\nc{\tP}{\widetilde{P}}
\nc{\tW}{\widetilde{W}}
\nc{\Dyn}{\mathrm{Dyn}}
\nc{\tD}{\widetilde{\Delta}}
\nc{\height}{{\operatorname{ht}}}
\nc{\bl}{\bigl(}
\nc{\br}{\bigr)}
\nc{\Hecke}{\mathrm{H}}
\nc{\HA}{\Hecke^{\mathrm{A}}}
\nc{\HB}{\Hecke^{\mathrm{B}}}
\newcommand{\scbul}{{\,\raise1pt\hbox{$\scriptscriptstyle\bullet$}\,}}
\nc{\vac}{{\phi}}
\nc{\Bt}{\B_\theta(\g)}
\nc{\be}{\begin{enumerate}}
\nc{\ee}{\end{enumerate}}
\nc{\low}{{\mathrm{low}}}
\nc{\upper}{{\mathrm{up}}}
\nc{\Zodd}{\Z_{\mathrm{odd}}}
\nc{\Ft}[1][n]{\mathbb{P}\mathrm{ol}_{#1}}
\nc{\Ftf}[1][n]{\widetilde{\mathbb{P}\mathrm{ol}}_{#1}}
\nc{\KA}{\on{K}^{\mathrm{A}}}
\nc{\KB}{\on{K}^{\mathrm{B}}}
\nc{\Res}{\on{Res}}
\nc{\Fc}[1][{n,m}]{\mathbf{F}_{#1}}
\nc{\tphi}{\tilde{\varphi}}
\nc{\CO}{\mathscr{O}}
\nc{\inte}{\mathrm{int}}
\nc{\Oint}{\mathcal{O}^{\ge0}_{\inte}}
\nc{\vs}{\vspace}
\nc{\tL}{\widetilde{L}}
\nc{\tu}{\tilde{u}}
\nc{\noi}{\noindent}
\nc{\heigh}{\mathfrak{t}}
\nc{\lowest}{\mathfrak{l}}
\nc{\rootl}{\mathsf{Q}}
\nc{\cl}{{\rm{cl}}}
\nc{\uqpg}{U'_q(\mathfrak g)}
\nc{\Oh}{\widehat{\mathcal{O}}}

\nc{\hV}{\widehat{V}}

\newenvironment{rouge}
{\color{red}}
{}

\nc{\bred}{\begin{rouge}}
\nc{\ered}{\end{rouge}}
\nc{\KLR}{quiver Hecke algebra}
\nc{\KLRs}{quiver Hecke algebras}
\nc{\cor}{\mathbf{k}}
\nc{\cora}{{\cor(A)}}
\nc{\haut}{\mathrm{ht}}
\nc{\tens}{\mathop\otimes}
\nc{\gmod}{\mbox{-$\mathrm{gmod}$}}
\nc{\proj}{\mbox{-$\mathrm{proj}$}}
\nc{\gproj}{\mbox{-$\mathrm{gproj}$}}
\nc{\smod}{\mbox{-$\mathrm{mod}$}}
\nc{\nmod}{\mbox{-$\mathrm{nilmod}$}}

\nc{\h}{\mathfrak h}
\nc{\Rnorm}{R^{\rm{norm}}}
\nc{\Runiv}{R^{\rm{univ}}}
\nc{\Vhat}{\widehat{V}}
\nc{\F}{\mathcal{F}}
\def\AA{{\mathcal A}}

\def\T{{\mathcal T}}

\nc{\fd}[1][A]{\on{\mathrm{flat.dim}_{#1}}}
\nc{\bP}{{\mathbb{P}}}
\nc{\bPh}{\widehat{\mathbb{P}}}
\nc{\bK}{\widehat{\mathbb{K}}}
\nc{\bV}[1][{n}]{\widehat{V}^{\otimes{#1}}}
\nc{\bVK}[1][{n}]{\widehat{V}^{\otimes{#1}}_K}
\nc{\opp}{\mathrm{opp}}
\nc{\col}{\colon}
\nc{\bnum}{\be[{\rm(i)}]}
\nc{\oep}{\epsilon}
\nc{\qtext}{\quad\text}
\nc{\qtextq}[1]{\quad\text{#1}\quad}
\nc{\longtwoheadrightarrow}[1][]{\xymatrix{\ar@{->>}[r]^-{{#1}}&}}
\nc{\epiTo}[1][]{\longtwoheadrightarrow[{#1}]}
\nc{\epito}{\twoheadrightarrow}
\nc{\monoTo}[1][]{\xymatrix{\ar@{>->}[r]^-{{#1}}&}}
\nc{\sym}{\mathfrak{S}}
\nc{\inp}[1]{{({#1})_{\mathrm{n}}}}
\nc{\rtl}{\rootl}
\nc{\wtd}{\widetilde}
\nc{\etens}{\boxtimes}
\nc{\ds}[1]{\mathrm{d}(#1)}
\nc{\rmat}[1]{{\mathbf r}_{\mspace{-2mu}\raisebox{-.5ex}{${\scriptstyle{#1}}$}}}
\nc{\shc}{\mathcal{C}}
\nc{\Fct}{{\on{Fct}}}
\nc{\tC}{\widetilde{\shc}}
\nc{\Zp}{\Z_{\ge0}}
\nc{\tPhi}{\widetilde{\Phi}}
\nc{\tT}{{\tilde{\T}}}
\nc{\Ob}{\on{Ob}}
\nc{\bwr}{\mbox{\large$\wr$}}
\nc{\Img}{\on{Im}}
\nc{\Ab}{\mathcal{A}^{\mathrm{big}}}
\nc{\Sb}{\mathcal{S}^{\mathrm{big}}}
\nc{\As}{\mathcal{A}}
\nc{\Ss}{\mathcal{S}}
\nc{\ntens}{\widetilde{\otimes}}
\nc{\hR}{\widehat{R}}
\nc{\nconv}{\star}
\nc{\ts}{\tilde{s}}
\nc{\sho}{\mathcal{O}}
\nc{\bc}{\begin{cases}}
\nc{\ec}{\end{cases}}

\nc{\UA}{U_q'(A^{(1)}_{N-1})}
\nc{\UAtwo}{U_q'(A^{(2)}_{N-1})}
\nc{\KR}{R_K}
\nc{\cQ}{\mathcal{Q}}
\nc{\Irr}{\mathcal{I}rr}
\nc{\tQ}{\widetilde{\cQ}}
\nc{\bs}{\mathbf{s}}
\nc{\bL}{\mathbb{L}}
\nc{\KP}{{\mathrm{KP}}}
\nc{\db}{\mathsf{b}^*}
\nc{\bfa}{{\mathbf{a}}}
\nc{\bfc}{{\mathbf{c}}}
\nc{\Po}{(P_\cl)_0}
\renewcommand{\Im}{\op{Im}}
\nc{\mono}{\rightarrowtail}
\nc{\tr}{\on{tr}}
\nc{\K}{\on{K}}
\nc{\FQ}[1]{\F_Q^{(#1)}}
\nc{\bQ}{\ol{Q}}
\nc{\conv}{\mathbin{\mbox{\large $\circ$}}}
\nc{\uqm}{\uqpg\smod}

\newlength{\mylength}
\title[Symmetric quiver Hecke algebras and R-matrices III]
{Symmetric quiver Hecke algebras and R-matrices of quantum affine
algebras III}

\author[S.-J. Kang, M. Kashiwara, M. Kim, S.-j. Oh]{Seok-Jin Kang$^{1}$, Masaki Kashiwara$^{2}$,  Myungho Kim and Se-jin Oh$^{3}$}

\address{Gwanak Wiberpolis 101-1601,
         Gwanak-Ro 195, Gwanak-Gu \\ Seoul 151-811, Korea}

         \email{soccerkang@hotmail.com}

\address{Research Institute for Mathematical Sciences \\
          Kyoto University \\ Kyoto 606-8502, Japan}

         \email{masaki@kurims.kyoto-u.ac.jp}

\address{School of Mathematics, Korea Institute for Advanced Study \\ Seoul 130-722, Korea}
         \email{mhkim@kias.re.kr}

\address{School of Mathematics, Korea Institute for Advanced Study \\ Seoul 130-722, Korea}
         \email{sejin092@gmail.com}

\thanks{$^{1}$This work was supported by NRF Grant  \# 2014-021261 and by NRF Grant \# 2013-055408}

\thanks{$^{2}$This work was partially supported by Grant-in-Aid for
Scientific Research (B) 22340005, Japan Society for the Promotion of
Science.}

\thanks{$^{3}$This work was supported by BK21 PLUS SNU Mathematical Sciences Division}



\date{}

\begin{document}

\begin{abstract}

Let $\CC^0_{\g}$ be the category of finite-dimensional integrable
modules over the quantum affine algebra $U_{q}'(\g)$ and let
$R^{A_\infty}\gmod$ denote the category of finite-dimensional graded
modules over the quiver Hecke algebra of type $A_{\infty}$. In this
paper, we investigate the relationship between the categories
$\CC^0_{A_{N-1}^{(1)}}$ and $\CC^0_{A_{N-1}^{(2)}}$ by constructing
the generalized quantum affine Schur-Weyl duality functors
$\F^{(t)}$ from $R^{A_\infty}\gmod$ to $\CC^0_{A_{N-1}^{(t)}}$
$(t=1,2)$.

\end{abstract}

\maketitle

\section*{Introduction}
The {\it quiver Hecke algebra} $R$, introduced independently by
Khovanov-Lauda \cite{KL09} and Rouquier \cite{R08}, provides a
categorification of the negative half $U^-_q(\mathsf{g})$ of a
quantum group $U_q(\mathsf{g})$. Moreover, its cyclotomic quotients
$R^\Lambda$, depending on  dominant integral weights $\Lambda$, also
provide  a categorification of the integrable highest weight modules
$V_q(\Lambda)$ over $U_q(\mathsf{g})$ \cite{KK11}. Recall that the
cyclotomic quotients of an affine Hecke algebra give a
categorification of integrable highest weight
$U(A^{(1)}_{N-1})$-modules. Thus the quiver Hecke algebras play the
role of affine Hecke algebras in the representation theory of {\it
all symmetrizable} quantum groups.

In \cite{CP96B,Che,GRV94}, Chari-Pressley, Cherednik and
Ginzburg-Reshetikhin-Vasserot constructed the {\it quantum affine
Schur-Weyl duality functor} that relates the category of
finite-dimensional modules over an affine Hecke algebra and the
category of finite-dimensional integrable
$U_{q}'(A_{N-1}^{(1)})$-modules. In \cite{KKK13A}, the first three
authors of this paper constructed a functor $\F$ from the category
of finite-dimensional graded modules over a {\em symmetric} quiver
Hecke algebra $R$ to the category of finite-dimensional integrable
modules over {\it any} quantum affine algebra $U_{q}'(\g)$. Here,
the quiver Hecke algebra $R$ is determined by a family of good
$U_q'(\g)$-modules. In this context, the quiver Hecke algebras can
be thought of as a generalization of affine Hecke algebras, which
gives the {\it generalized quantum affine Schur-Weyl duality
functor} $\F$.

The representation theory of quantum affine algebras  has been
extensively  investigated with various approaches (see, for example,
\cite{CP95,FR99,Her101,Kas02,Nak041}). By the work of \cite{KKK13A},
we propose a new approach for studying the representations of
quantum affine algebras through the representation theory of quiver
Hecke algebras.

Let $\CC^0_{\g}$ denote the category of finite-dimensional
integrable modules over the quantum affine algebra $U_{q}'(\g)$, and
let $R^{A_\infty}\gmod$ denote the category of finite-dimensional
graded modules over the quiver Hecke algebra of type $A_{\infty}$.
The purpose of this paper is to investigate the relationship between
the categories $\CC^0_{A_{N-1}^{(1)}}$ and $\CC^0_{A_{N-1}^{(2)}}$
(see Table~\ref{table:Dynkin}) by constructing exact functors
$\F^{(t)}:R^{A_\infty}\gmod \rightarrow \CC^0_{A_{N-1}^{(t)}}$
$(t=1,2)$ (see also \cite{Her101}).

$$
\xymatrix@R=5ex@C=5ex{ & R^{A_\infty}\gmod \ar[dl]_{\mathcal{F}^{(1)}} \ar[dr]^{\mathcal{F}^{(2)}} \\
\mathscr{C}^0_{A^{(1)}_{N-1}} \ar@{<.>}[rr] &&
\mathscr{C}^0_{A^{(2)}_{N-1}.} }$$

To construct such functors, we first choose a family of good
$U_{q}'(\g)$-modules and study the distribution of poles of
normalized $R$-matrices between them. Then by the general argument
given in \cite{KKK13A}, we obtain the generalized quantum affine
Schur-Weyl duality functor $\F: R\gmod \rightarrow \CC^0_{\g}$. In
particular, it was shown in \cite{KKK13A} that the family of good
$U_{q}'(A_{N-1}^{(1)})$-modules $\{V(\varpi_{1})_{q^{2s}} \mid s \in
\Z \}$ gives the functor $\F^{(1)}$.

In this paper, based on the results of \cite{Oh14} on the normalized
$R$-matrices of $U'_q(A^{(2)}_{N-1})$-modules,  we prove that  the
family of good $U'_q(A^{(2)}_{N-1})$-modules $\{
V(\varpi_1)_{q^{2s}} \mid s \in \Z \}$   yields a quiver whose
underlying graph is of type $A_\infty$, and  then we construct the
exact functor  $\F^{(2)}\col R^{A_\infty}\gmod \rightarrow
\mathscr{C}^0_{A^{(2)}_{N-1}}$. Through the exact functors
$\F^{(t)}$ $(t=1,2)$, one can observe that the categories
$\mathscr{C}^0_{A^{(1)}_{N-1}}$  and $\mathscr{C}^0_{A^{(2)}_{N-1}}$
have many similar properties (for example, see Proposition
\ref{prop:CP} and Corollary \ref{cor:surjhoms}). Note that some of
these similarities have been already observed by Hernandez
\cite{Her101} by a different approach.

We prove that the functor $\F^{(2)}\col R^{A_\infty}\gmod
\rightarrow \mathscr{C}^0_{A^{(2)}_{N-1}}$ factors through a certain
localization
 $\T_N$.
Furthermore,  the induced functor
$\tF^{(2)}\col\T_N\to\CC^0_{A^{(2)}_{N-1}}$ gives a ring isomorphism
$$K(\T_N)/(q-1)K(\T_N) \isoto K(\CC^0_{A^{(2)}_{N-1}})$$
as in the case of the Grothendieck ring $K(\CC^0_{A^{(1)}_{N-1}})$
in \cite{KKK13A}. Hence we obtain the diagram
$$
\xymatrix@R=4ex@C=7ex{ & R^{A_\infty}\gmod \ar[ddl]_{\mathcal{F}^{(1)}} \ar[ddr]^{\mathcal{F}^{(2)}} \ar@{->>}[d] \\
& \T_N \ar[dl]^{\tF^{(1)}} \ar[dr]_{\tF^{(2)}} \\
\mathscr{C}^0_{A^{(1)}_{N-1}} \ar@{<.>}[rr] &&
\mathscr{C}^0_{A^{(2)}_{N-1}}, }$$ where $\tF^{(t)}$ gives a
bijection between the simple modules (up to degree shift and
isomorphism) in $\T_N$ and the  simple modules (up to isomorphism)
in $\mathscr{C}^0_{A^{(t)}_{N-1}}$ $(t=1,2)$. With this approach, we
prove that the induced functors $\tF^{(1)}$ and $\tF^{(2)}$ give the
correspondence between the simple modules in
$\mathscr{C}^0_{A^{(1)}_{N-1}}$ and the  simple modules in
$\mathscr{C}^0_{A^{(2)}_{N-1}}$, which preserves their dimensions
(Theorem~\ref{th: dim}).

Let us compare this with  one of the results in \cite{Her101}.
First, Hernandez defined the {\em twisted $q$-character homomorphism
$\chi_q^\sigma$} from the Grothendieck ring of finite-dimensional
integrable modules over a twisted quantum affine algebra to a
certain polynomial ring. Note that it is an analogue of the
$q$-character homomorphism $\chi_q$ for untwisted quantum affine
algebras in \cite{FR99}. Then he found a ring homomorphism from the
codomain of $\chi_q$ to the codomain of $\chi_q^\sigma$, which
induces an isomorphism between the image of the $\chi_q$ and that of
$\chi_q^\sigma$. During its proof, he showed that this isomorphism
sends the $q$-characters of Kirillov-Reshetikhin modules to the
twisted $q$-characters of Kirillov-Reshetikhin modules. But it is
not known whether the isomorphism sends the $q$-characters of simple
modules to the twisted $q$-characters of simple modules or not. We
expect that the isomorphism in \cite[Theorem 4.15]{Her101} coincides
with ours in Corollary \ref{cor: iso bwn gro ring}. Since the
results in \cite{Her101} cover not only type $A$ but also other
types, one may expect that there are similar correspondences between
untwisted and twisted quantum affine algebras of other types. Our
another paper \cite{KKKO15} is initiated by this observation and
provides a correspondence between certain subcategories of
$\mathscr{C}^0_\g$ over untwisted and twisted quantum affine
algebras of type $A$ and $D$ through the representation theory of
quiver Hecke algebras.

This paper is organized as follows. In Section
\ref{sec:backgrounds}, we briefly review the results of
\cite{KKK13A} on the generalized quantum affine Schur-Weyl duality
functors. In Section \ref{sec:comparison}, we  compare the
denominators of normalized $R$-matrices and the homomorphisms
between fundamental representations over $U_q(A^{(2)}_{N-1})$. This
comparison provides the main ingredients for the construction of
exact functors $\F^{(t)}$ $(t=1,2)$. In Section \ref{sec:F2}, we
construct the exact functor $\F^{(2)}: R^{A_\infty}\gmod \rightarrow
\mathscr{C}^0_{A^{(2)}_{N-1}}$ and investigate the relationship
between the categories $\mathscr{C}^0_{A^{(t)}_{N-1}}$ $(t=1,2)$ via
 $R^{A_\infty}\gmod$.

\section*{Convention}

\bnum
\item All the algebras and rings in this paper are assumed to have a unit,
and modules over them are unitary.
\item For a ring $A$,  an $A$-module means a left $A$-module.
\item For a statement $P$, $\delta(P)$ is $1$ if $P$ is true and $0$ if $P$ is false.
\item For a ring $A$, we denote by $\Mod(A)$ the category of $A$-modules.
When $A$ is an algebra over a field $\cor$, we denote by $A\smod$
the category of $A$-modules which are finite-dimensional over
$\cor$.

If $A$ is a graded ring, then we denote by $\Modg(A)$, $A\gmod$
their graded version with homomorphism preserving the grading as
morphisms. They are also abelian categories.
\item For a ring $A$, we denote by $A^\times$ the set of
invertible elements of $A$.
\item For an abelian category $\mathcal C$, we denote by $K(\mathcal C)$ the Grothendieck group of $\mathcal C$.
\ee

\section{Symmetric quiver Hecke algebras and quantum affine algebras} \label{sec:backgrounds}
\subsection{Cartan datum and quantum groups}\label{subsec:Cartan}
In this subsection, we recall the definition of quantum groups. Let
$I$ be an index set. A \emph{Cartan datum} is a quintuple $(A,P,
\Pi,P^{\vee},\Pi^{\vee})$ consisting of
\begin{enumerate}[(a)]
\item an integer-valued matrix $A=(a_{ij})_{i,j \in I}$,
called \emph{the symmetrizable generalized Cartan matrix},
 which satisfies
\bni
\item $a_{ii} = 2$ $(i \in I)$,
\item $a_{ij} \le 0 $ $(i \neq j)$,
\item $a_{ij}=0$ if $a_{ji}=0$ $(i,j \in I)$,
\item there exists a diagonal matrix
$D=\text{diag} (\mathsf s_i \mid i \in I)$ such that $DA$ is
symmetric and $\mathsf s_i$ are positive integers,
\end{enumerate}

\item a free abelian group $P$, called the \emph{weight lattice},
\item $\Pi= \{ \alpha_i \in P \mid \ i \in I \}$, called
the set of \emph{simple roots},
\item $P^{\vee}\seteq\Hom(P, \Z)$, called the \emph{co-weight lattice},
\item $\Pi^{\vee}= \{ h_i \ | \ i \in I \}\subset P^{\vee}$, called
the set of \emph{simple coroots},
\end{enumerate}

\noindent satisfying the following properties:
\begin{enumerate}
\item[(i)] $\langle h_i,\alpha_j \rangle = a_{ij}$ for all $i,j \in I$,
\item[(ii)] $\Pi$ is linearly independent,
\item[(iii)] for each $i \in I$, there exists $\Lambda_i \in P$ such that
           $\langle h_j, \Lambda_i \rangle =\delta_{ij}$ for all $j \in I$.
\end{enumerate}
We call $\Lambda_i$ the \emph{fundamental weights}. The free abelian
group $\rootl\seteq\soplus_{i \in I} \Z \alpha_i$ is called the
\emph{root lattice}. Set $\rootl^{+}= \sum_{i \in I} \Z_{\ge 0}
\alpha_i\subset\rootl$ and $\rootl^{-}= \sum_{i \in I} \Z_{\le0}
\alpha_i\subset\rootl$. For $\beta=\sum_{i\in I}m_i\al_i\in\rootl$,
we set $|\beta|=\sum_{i\in I}|m_i|$.

Set $\mathfrak{h}=\Q \otimes_\Z P^{\vee}$. Then there exists a
symmetric bilinear form  $( \ , \ )$  on $\mathfrak{h}^*$ satisfying
$$ (\alpha_i , \alpha_j) =\mathsf s_i a_{ij} \quad (i,j \in I)
\quad\text{and $\lan h_i,\lambda\ran=
\dfrac{2(\alpha_i,\lambda)}{(\alpha_i,\alpha_i)}$ for any
$\lambda\in\mathfrak{h}^*$ and $i \in I$}.$$

Let $q$ be an indeterminate. For each $i \in I$, set $q_i =
q^{\,\mathsf s_i}$.

\begin{definition} \label{def:qgroup}
The {\em quantum group} $U_q(\g)$  with  a Cartan datum
$(A,P,\Pi,P^{\vee}, \Pi^{\vee})$ is the  algebra  over $\mathbb
Q(q)$ generated by $e_i,f_i$ $(i \in I)$ and $q^{h}$ $(h \in
P^\vee)$ satisfying following relations:
\begin{equation*}
\begin{aligned}
& q^0=1,\quad q^{h} q^{h'}=q^{h+h'}\ \text{for $h,h' \in P$,}\\
& q^{h}e_i q^{-h}= q^{\lan h, \alpha_i\ran} e_i, \quad
        q^{h}f_i q^{-h} = q^{-\lan h, \alpha_i\ran} f_i \quad
        \text{for $ h \in P^\vee$, $i \in I$,} \\
& e_if_j - f_je_i = \delta_{ij} \dfrac{K_i -K^{-1}_i}{q_i- q^{-1}_i
}\quad\text{where $K_i=q^{\mathsf s_i h_i}$, }\\
& \sum^{1-a_{ij}}_{r=0} (-1)^r
        e^{(1-a_{ij}-r)}_i e_j e^{(r)}_i
 =\sum^{1-a_{ij}}_{r=0} (-1)^r
        f^{(1-a_{ij}-r)}_if_j f^{(r)}_i=0 \quad \text{ if } i \ne j.
\end{aligned}
\end{equation*}
\end{definition}

Here, we set $[n]_i =\dfrac{ q^n_{i} - q^{-n}_{i} }{ q_{i} -
q^{-1}_{i} },\quad
  [n]_i! = \prod^{n}_{k=1} [k]_i$, $e_i^{(m)}=\dfrac{e_i^n}{[n]_i!}$ and $f_i^{(n)}=\dfrac{f_i^n}{[n]_i!}$
  for all $n \in Z_{\ge 0}$, $i \in I$.

\subsection{Quiver Hecke algebras}
We recall the definition of \KLRs\ associated with a given Cartan
datum $(A, P, \Pi, P^{\vee}, \Pi^{\vee})$.

Let $\cor$ be a commutative ring. For $i,j\in I$ such that
$i\not=j$, set Let us take  a family of polynomials
$(Q_{ij})_{i,j\in I}$ in $\cor[u,v]$ which are of the form
\begin{equation} \label{eq:Q}
Q_{ij}(u,v) =
\delta(i \ne j)\sum\limits_{ \substack{ (p,q)\in \Z^2_{\ge0} \\
(\al_i , \al_i)p+(\al_j , \al_j)q=-2(\al_i , \al_j)}} t_{i,j;p,q}
u^p v^q
\end{equation}
with $t_{i,j;p,q}\in\cor$, $t_{i,j;p,q}=t_{j,i;q,p}$ and
$t_{i,j:-a_{ij},0} \in \cor^{\times}$. Thus we have
$Q_{i,j}(u,v)=Q_{j,i}(v,u)$.

We denote by $\sym_{n} = \langle s_1, \ldots, s_{n-1} \rangle$ the
symmetric group on $n$ letters, where $s_i\seteq (i, i+1)$ is the
transposition of $i$ and $i+1$. Then $\sym_n$ acts on $I^n$ by place
permutations.

For $n \in \Z_{\ge 0}$ and $\beta \in \rootl^+$ such that $|\beta| =
n$, we set
$$I^{\beta} = \set{\nu = (\nu_1, \ldots, \nu_n) \in I^{n}}%
{ \alpha_{\nu_1} + \cdots + \alpha_{\nu_n} = \beta }.$$

\begin{definition}
For $\beta \in \rootl^+$ with $|\beta|=n$, the {\em
Khovanov-Lauda-Rouquier algebra}  $R(\beta)$  at $\beta$ associated
with a Cartan datum $(A,P, \Pi,P^{\vee},\Pi^{\vee})$ and a matrix
$(Q_{ij})_{i,j \in I}$ is the $\cor$-algebra generated by the
elements $\{ e(\nu) \}_{\nu \in  I^{\beta}}$, $ \{x_k \}_{1 \le k
\le n}$, $\{ \tau_m \}_{1 \le m \le n-1}$ satisfying the following
defining relations:
\begin{equation*} \label{eq:KLR}
\begin{aligned}
& e(\nu) e(\nu') = \delta_{\nu, \nu'} e(\nu), \ \
\sum_{\nu \in  I^{\beta} } e(\nu) = 1, \\
& x_{k} x_{m} = x_{m} x_{k}, \ \ x_{k} e(\nu) = e(\nu) x_{k}, \\
& \tau_{m} e(\nu) = e(s_{m}(\nu)) \tau_{m}, \ \ \tau_{k} \tau_{m} =
\tau_{m} \tau_{k} \ \ \text{if} \ |k-m|>1, \\
& \tau_{k}^2 e(\nu) = Q_{\nu_{k}, \nu_{k+1}} (x_{k}, x_{k+1})
e(\nu), \\
& (\tau_{k} x_{m} - x_{s_k(m)} \tau_{k}) e(\nu) = \begin{cases}
-e(\nu) \ \ & \text{if} \ m=k, \nu_{k} = \nu_{k+1}, \\
e(\nu) \ \ & \text{if} \ m=k+1, \nu_{k}=\nu_{k+1}, \\
0 \ \ & \text{otherwise},
\end{cases} \\
& (\tau_{k+1} \tau_{k} \tau_{k+1}-\tau_{k} \tau_{k+1} \tau_{k}) e(\nu)\\
& =\begin{cases} \dfrac{Q_{\nu_{k}, \nu_{k+1}}(x_{k}, x_{k+1}) -
Q_{\nu_{k}, \nu_{k+1}}(x_{k+2}, x_{k+1})} {x_{k} - x_{k+2}}e(\nu) \
\ & \text{if} \
\nu_{k} = \nu_{k+2}, \\
0 \ \ & \text{otherwise}.
\end{cases}
\end{aligned}
\end{equation*}
\end{definition}

The above relations become homogeneous by assigning
\begin{equation*} \label{eq:Z-grading}
\deg e(\nu) =0, \quad \deg\, x_{k} e(\nu) = (\alpha_{\nu_k} ,
\alpha_{\nu_k}), \quad\deg\, \tau_{l} e(\nu) = - (\alpha_{\nu_l} ,
\alpha_{\nu_{l+1}}),
\end{equation*}
and hence $R( \beta )$ is $\Z$-gradable.

For an element $w$ of the symmetric group $\sym_n$, let us choose  a
reduced expression $w=s_{i_1}\cdots s_{i_\ell}$, and set
$$\tau_w=\tau_{i_1}\cdots \tau_{i_\ell}.$$ In general, it depends on
the choice of reduced expressions of $w$. Then we have
$$R(\beta)=\soplus_{\nu\in  I^{\beta} , \;w\in\sym_n}\cor[x_1,\ldots, x_n]e(\nu)\tau_w.$$

 For a graded $R(\beta)$-module $M=\bigoplus_{k \in \Z} M_k$, we define
$qM =\bigoplus_{k \in \Z} (qM)_k$, where
 \begin{align*}
 (qM)_k = M_{k-1} & \ (k \in \Z).
 \end{align*}
We call $q$ the \emph{grade shift functor} on the category of graded
$R(\beta)$-modules.

\medskip
For $\beta, \gamma \in \rootl^+$ with $|\beta|=m$, $|\gamma|= n$,
 set
$$e(\beta,\gamma)=\displaystyle\sum_{\substack%
{\nu \in I^{m+n}, \\ (\nu_1, \ldots ,\nu_m) \in I^{\beta},\\
(\nu_{m+1}, \ldots ,\nu_{m+n}) \in I^{\gamma}}} e(\nu) \in
R(\beta+\gamma).$$ Then $e(\beta,\gamma)$ is an idempotent. Let \eq
R( \beta)\tens R( \gamma  )\to e(\beta,\gamma)R(
\beta+\gamma)e(\beta,\gamma) \label{eq:embedding} \eneq be the
$\cor$-algebra homomorphism given by
\begin{equation*}
\begin{aligned}
& e(\mu)\tens e(\nu)\mapsto e(\mu*\nu) \ \ (\mu\in I^{\beta}),\\
& x_k\tens 1\mapsto x_ke(\beta,\gamma) \ \  (1\le k\le m), \\
& 1\tens x_k\mapsto x_{m+k}e(\beta,\gamma) \ \  (1\le k\le n), \\
& \tau_k\tens 1\mapsto \tau_ke(\beta,\gamma) \ \  (1\le k<m), \\
& 1\tens \tau_k\mapsto \tau_{m+k}e(\beta,\gamma) \ \  (1\le k<n,
\end{aligned}
\end{equation*}
where  $\mu*\nu$ is the concatenation of $\mu$ and $\nu$; i.e.,
$\mu*\nu=(\mu_1,\ldots,\mu_m,\nu_1,\ldots,\nu_n)$.

\medskip
For a  $R(\beta)$-module $M$ and a  $R(\gamma)$-module $N$, we
define the \emph{convolution product} $M\conv N$ by
$$M\conv N \seteq R(\beta + \gamma) e(\beta,\gamma)
\tens_{R(\beta )\otimes R( \gamma)}(M\otimes N). $$

\subsection{$R$-matrices with spectral parameters}
 For $|\beta|=n$ and $1\le a<n$,  we define $\vphi_a\in R( \beta)$
by \eq&&\ba{l}
  \vphi_a e(\nu)=
\begin{cases}
  \bl\tau_ax_a-x_{a}\tau_a\br e(\nu) & \text{if $\nu_a=\nu_{a+1}$,} \\[2ex]
\tau_ae(\nu)& \text{otherwise.}
\end{cases}
\label{def:int} \ea \eneq They are called the {\em intertwiners}.
Since $\{\vphi_k\}_{1\le k\le n-1}$ satisfies the braid relation, we
have a well-defined element $\vphi_w \in R(\beta)$ for each $w \in
\sym_n$.

 For $m,n\in\Z_{\ge0}$,
we set \eqn &&\sym_{m,n} \seteq \set{w\in\sym_{m+n}}{\text{
$w(i)<w(i+1)$ for any $i\not=m$}}. \eneqn For example, \eqn
&&w[{m,n}](k)=\begin{cases}k+n&\text{if $1\le k\le m$,}\\
k-m&\text{if $m<k\le m+n$.}\end{cases} \eneqn is an element in
$\sym_{m,n}$.

Let $\beta,\gamma\in \rtl^+$ with $|\beta|=m$, $|\gamma|=n$ and let
$M$ be an $R(\beta)$-module and $N$ an $R(\gamma)$-module. Then the
map
$$M\tens N\to q^{(\beta,\gamma)-2\inp{\beta,\gamma}}N\conv M$$ given by
$$u\tens v\longmapsto \vphi_{w[n,m]}(v\tens u)$$
is an  $R( \beta,\gamma )$-module homomorphism by \cite[Lemma
1.3.1]{KKK13A}, and it extends to an $R( \beta +\gamma)$-module
homomorphism \eq &&R_{M,N}\col M\conv N\To
q^{(\beta,\gamma)-2\inp{\beta,\gamma}}N\conv M, \eneq where the
symmetric bilinear form $\inp{{\scbul,\scbul}}$ on $\rtl$ is given
by $\inp{\al_i,\al_j}=\delta_{ij}.$

\begin{definition} A \KLR\ $R( \beta )$ is {\em symmetric} if $Q_{i,j}(u,v)$ is a
polynomial in $\cor[u-v]$ for all $i,j \in \supp{\beta}$. Here
$\supp(\beta) = \set{i \in I}{ n_i \ne 0 \text{ for } \beta =
\sum_{i \in I}n_i\alpha_i }.$ \end{definition}

\medskip
From now on, we assume that \KLRs\ are symmetric. Let $z$ be an
indeterminate  which is  homogeneous of degree $2$, and let $\psi_z$
be the algebra homomorphism \eqn &&\psi_z\col R( \beta )\to
\cor[z]\tens R( \beta ) \eneqn given by
$$\psi_z(x_k)=x_k+z,\quad\psi_z(\tau_k)=\tau_k, \quad\psi_z(e(\nu))=e(\nu).$$

For an $R( \beta )$-module $M$, we denote by $M_z$ the
$\bl\cor[z]\tens R( \beta )\br$-module $\cor[z]\tens M$ with the
action of $R( \beta )$ twisted by $\psi_z$. Namely,
\begin{equation}\label{eq:spt1}
\begin{aligned}
& e(\nu)(a\tens u)=a\tens e(\nu)u, \\
& x_k(a\tens u)=(za)\tens u+a\tens (x_ku), \\
& \tau_k(a\tens u)=a\tens(\tau_k u)
\end{aligned}
\end{equation}
for  $\nu\in I^\beta$,  $a\in \cor[z]$ and $u\in M$. For $u\in M$,
we sometimes denote by $u_z$ the corresponding element $1\tens u$ of
the $R( \beta )$-module  $M_z$.

For a non-zero $R(\beta)$-module $M$ and a non-zero
$R(\gamma)$-module $N$,
\eq&&\parbox{70ex}{%
let $s$ be the order of zeroes of $R_{M_z,N_{z'}}\col  M_z\conv
N_{z'}\To q^{(\beta,\gamma)-2\inp{\beta,\gamma}}N_{z'}\conv M_z$;
i.e., the largest non-negative integer such that the image of
$R_{M_z,N_{z'}}$ is contained in $(z'-z)^s
q^{(\beta,\gamma)-2\inp{\beta,\gamma}}N_{z'}\conv M_z$.}
\label{def:s} \eneq Note that \cite[Proposition 1.4.4 (iii)]{KKK13A}
shows that
 such an $s$ exists   and
 $s\le \inp{\beta,\gamma}$.

\begin{definition} For a non-zero $R(\beta)$-module $M$ and a non-zero
$R(\gamma)$-module $N$, we set
$$\ds{M,N}\seteq(\beta,\gamma)-2\inp{\beta,\gamma}+2s,$$ and define
$$\rmat{M,N}\col M\conv N\to q^{\ds{M,N}}N\conv M$$
by  $$\rmat{M,N} = \bl (z'-z)^{-s}R_{M_z,N_{z'}}\br\vert_{z=z'=0}.$$
\end{definition}
 By \cite[Proposition 1.4.4 (ii)]{KKK13A}, the morphism $\rmat{M,N}$ does not vanish
if $M$ and $N$ are non-zero.

For $\beta_1,\ldots,\beta_t\in\rtl^+$, a sequence of
$R(\beta_k)$-modules $M_k$ ($k=1,\ldots,t$) and $w\in \sym_t$, we
set $d=\sum\limits\ds{M_i,M_{j}}$, where the summation ranges over
the set
$$\set{(i,j)}{1\le i<j\le t, w(i)>w(j)}.$$
We define \eq &&\rmat{M_1,\ldots,M_t}^{\quad w} =\rmat{\{M_s\}_{1\le
s\le t}}^{\quad w}\col M_1\conv \cdots\conv M_t \to q^dM_{w(1)}\conv
\cdots\conv M_{w(t)} \eneq by induction on the length of $w$ as
follows: \eqn &&\hs{-3.5ex}\rmat{\{M_a\}_{1\le a\le t}}^{\ w}
=\begin{cases}
\id_{M_1\conv \cdots\conv M_t}&\text{if $w=e$,}\\[2ex]
\ba{l}\hs{-1ex}\rmat{\{M_{s_k(a)}\}_{1\le a\le t}}^{\ ws_k} \circ
\bl M_{1}\conv \cdots \conv M_{k-1}\conv \\ \hs{3ex}
 \rmat{M_{k}, M_{k+1}}  \conv  M_{k+2}\conv\cdots\conv M_{t}\br\ea
&\text{if $w(k)>w(k+1)$.}
\end{cases}
\eneqn Then it does not depend on the choice of $k$ and
$\rmat{M_1,\ldots,M_t}^{\ w}$ is well-defined, because the
homomorphisms $\rmat{M,N}$ satisfy the Yang-Baxter
equation(\cite[\S1.4]{KKK13A}).

Similarly, we define \eq R_{M_1,\ldots, M_t}^w\col
 M_1\conv \cdots \conv M_t
\to q^bM_{w(1)}\conv \cdots \conv M_{w(t)}, \eneq where
$b=\sum\limits_{ \substack{1\le k<k'\le t,\\ w(k)>w(k')} }
(\beta_k,\beta_{k'})-2\inp{\beta_k,\beta_{k'}}$.

We set \eq &&\rmat{M_1,\ldots,
M_t}\seteq\rmat{M_1,\ldots,M_t}^{\hs{2ex} w_t} \quad \text{and}
\quad R_{M_1,\ldots, M_t}\seteq R_{M_1,\ldots, M_t}^{\hs{2ex} w_t},
\eneq where $w_t$ is the longest element of $\sym_t$.

\subsection{Quantum affine algebras}
In this subsection, we briefly review the representation theory of
quantum affine algebras following \cite{AK,Kas02}. When concerned
with quantum affine algebras, we take the algebraic closure of
$\C(q)$ in $\cup_{m >0}\C((q^{1/m}))$ as  the  base field $\cor$.

 We choose $0\in I$ as the leftmost vertices in the tables
in \cite[{pages 54, 55}]{Kac} except $A^{(2)}_{2n}$-case in which
case we take the longest simple root as $\al_0$. Set $I_0
=I\setminus\{0\}$.

The weight lattice $P$ is given by
\begin{align*}
  P = \Bigl(\soplus_{i\in I}\Z \La_i\Bigr) \oplus \Z \delta,
\end{align*}
and the simple roots are given by
$$\al_i=\sum_{j\in I}a_{ji}\La_j+\delta(i=0)\delta.$$
The weight $\delta$ is called the imaginary root. There exist
$d_i\in\Z_{>0}$ such that
$$\delta=\sum_{i\in I}d_i\al_i.$$
Note that $d_i=1$ for $i=0$. The simple coroots $h_i\in
P^\vee\seteq\Hom_{\Z} (P,\Z)$ are given by \eqn && \lan
h_i,\La_j\ran=\delta_{ij},\quad \lan h_i,\delta\ran=0. \eneqn Hence
we have $\lan h_i,\al_j\ran=a_{ij}$.

Let $c=\sum_{i\in I}c_i h_i$ be a unique element such that
$c_i\in\Z_{>0}$ and
$$\Z\, c=\set{h\in\soplus\nolimits_{i\in I}\Z h_i}{\text{$\lan h,\al_i\ran=0$ for any $i\in I$}}.$$
Let us take a $\Q$-valued symmetric bilinear form  $(\scbul,\scbul)$
on $P$ such that
$$\text{
$\lan h_i,\la\ran=\dfrac{2(\al_i,\la)}{(\al_i,\al_i)}$ and $(\delta,
\la)=\lan c,\la\ran $ for any $\la\in P$.}
$$
Let $q$ be an indeterminate. For each $i \in I$, set $q_i =
q^{(\al_i,\al_i)/2}$.

Let us denote by $U_q(\g)$ the quantum group associated with the
affine Cartan datum $(A,P, \Pi,P^{\vee},\Pi^{\vee})$.
 We denote by $\uqpg$ the subalgebra of $U_q(\mathfrak{g})$ generated by $e_i,f_i,K_i^{\pm1}(i=0,1,,\ldots,n)$. We call $\uqpg$ the \emph{quantum affine algebra} associated with $(A,P, \Pi,P^{\vee},\Pi^{\vee})$.

The algebra $\uqpg$ has a Hopf algebra structure with the coproduct:
\begin{align}
\Delta(K_i)=K_i\tens K_i, \quad \Delta(e_i)=e_i\tens K_i^{-1}+1\tens
e_i,\quad \Delta(f_i)=f_i\tens 1+K_i\tens f_i.
\end{align}

Set
$$P_\cl=P/\Z \delta$$
and call it the {\em classical weight lattice}. Let $\cl\col P\to
P_\cl$ be the projection. Then $P_\cl=\soplus\nolimits_{i\in
I}\Z\,\cl(\La_i)$. Set $P^0_\cl = \set{\la\in P_\cl}{\lan
c,\la\ran=0}\subset P_\cl$.

A $\uqpg$-module $M$ is called an {\em integrable module} if

\hs{5ex}\parbox[t]{70ex}{
 \bna
\item $M$ has a weight space decomposition
$$M = \bigoplus_{\lambda \in P_\cl} M_\lambda,$$
where  $M_{\lambda}= \set{ u \in M }{%
\text{$K_i u =q_i^{\lan h_i , \lambda \ran} u$ for all $i\in I$}}$,
\item  the actions of
 $e_i$ and $f_i$ on $M$ are locally nilpotent for any $i\in I$.
\ee}

Let us denote by $\CC_\g$ the abelian tensor category of
finite-dimensional integrable $\uqpg$-modules.

If $M$ is a  simple module in $\CC_\g$, then there exists a non-zero
vector $u\in M$ of weight $\la\in P_\cl^0$ such that $\la $ is
dominant (i.e., $\lan h_i,\la\ran\ge0$ for any $i\in I_0$) and all
the weights of $M$ lie in $\la-\sum_{i\in I_0}\Z_{\ge0}\al_i$. We
say that $\la$ is the {\em dominant extremal weight} of $M$ and $u$
is a {\em dominant extremal vector} of $M$. Note that a dominant
extremal vector of $M$ is unique up to a constant multiple.

Let $M$ be an integrable $\uqpg$-module. Then the {\em affinization}
$M_\aff$ of $M$ is the   $P$-graded $\uqpg$-module
$$M_\aff=\soplus_{\la\in P}(M_\aff)_{\la}\qtext{with $(M_\aff)_{\la}=M_{\cl(\la)}$.}
$$
Let us denote by  $\cl\col M_\aff \to M$ the canonical $\cor$-linear
homomorphism. The actions  $$e_i\col (M_\aff)_\la  \to
(M_\aff)_{\la+\al_i}
 \quad \text{and} \quad f_i\col(M_\aff)_\la\to (M_\aff)_{\la-\al_i}$$ are defined
  in a way  that
they commute with $\cl\col M_\aff\to M$.

  We  denote by $z_M\col M_\aff\to M_\aff$ the $\uqpg$-module
automorphism of weight $\delta$ defined by $(M_\aff)_\la
\overset{\sim}{\to} M_{\cl(\la)} \overset{\sim}{\to}
(M_\aff)_{\la+\delta}$. For $x \in \cor^\times$, we define
 $$M_x   \seteq M_\aff/(z_M-x)M_\aff.$$

We embed $P_\cl$ into $P$ by $\iota\col P_\cl\to P$ which is given
by $\iota(\cl(\La_i))=\La_i$. For $u\in M_\la$ ($\la\in P_\cl$), let
us denote by $u_z\in (M_\aff)_{\iota(\la)}$ the element such that
$\cl(u_z)=u$. With this notation,  we have
$$e_i(u_z)=z^{\delta_{i,0}}(e_iu)_z, \quad
f_i(u_z)=z^{-\delta_{i,0}}(f_iu)_z, \quad K_i(u_z)=(K_iu)_z.
$$
Then we have $M_\aff\simeq\cor[z,z^{-1}]\tens M$.

 Let $M$ be an integrable  $U_q(\g)$-module.
  A weight vector $u \in M_\lambda \ (\lambda \in P)$ is called
an {\em extremal vector} if there exists  a family of vectors
$\{u_w\}_{w\in W}$ satisfying the following properties: \eq
&&\text{$u_w=u$ for $w=e$,} \nonumber \\
&& \hbox{if $\langle h_i,w\lambda\rangle\ge 0$, then
$e_iu_w=0$ and $f_i^{(\langle h_i,w\lambda\rangle)}u_w=u_{s_iw}$,} \nonumber \\
&&\hbox{if $\langle h_i,w\lambda\rangle\le 0$, then $f_iu_w=0$ and
$e_i^{( -  \langle h_i,w\lambda\rangle)} u_w=u_{s_iw}$.} \nonumber
\eneq

If such $\{u_w\}_{w\in W}$ exists, then it is unique and $u_w$ has
weight $w\lambda$.

For $\lambda \in P$, let us denote by $W(\lambda)$ the
$U_q(\g)$-module generated by $u_\lambda$ with the defining relation
that $u_\lambda$ is an extremal vector of weight $\lambda$ (see
\cite{Kas94}). This is in fact a set of infinitely many linear
relations on $u_\lambda$.

 Set $\varpi_i=\gcd(c_0,c_i)^{-1}(c_0\Lambda_i-c_i\Lambda_0) \in P$
for $i=1,2,\ldots,n$. Then $\{\cl(\varpi_i)\}_{i=1,2,\ldots,n}$
forms a basis of $P^0_\cl$. We call $\varpi_i$ a {\em level $0$
fundamental weight}.
 As shown in \cite{Kas02}, for each $i=1,\ldots,n$, there exists a $\uqpg$-module automorphism
$z_i \col W(\varpi_i) \rightarrow W(\varpi_i)$ which sends
$u_{\varpi_i}$ to $u_{\varpi_i + \mathsf{d_i} \delta}$, where
$\mathsf{d_i} \in \Z_{>0}$ denotes the generator of the free abelian
group $\set{ m \in \Z}{\varpi_i + m \delta \in W \varpi_i }$.

We define the $\uqpg$-module $V(\varpi_i)$ by
$$V(\varpi_i) = W(\varpi_i) / (z_i-1) W(\varpi_i).$$
We call $V(\varpi_i)$ the {\em fundamental representation of $\uqpg$
of weight $\varpi_i$}. We have $V(\varpi_i)_\aff \simeq \cor[z_i^{{1
/ \mathsf{d_i}}}] \otimes_{\cor[z_i]} W(\varpi_i)$.

\bigskip
 If a $\uqpg$-module $M \in \CC_\g$ has a {\it  bar
involution}, a crystal basis with {\it simple crystal graph} and a
{\it lower global basis}, then we say that $M$ is a {\em good
module}.
 For the precise definition, see \cite[\S\;8]{Kas02}.
 For example, the fundamental representation
$V(\varpi_i)$ is a good $\uqpg$-module. Every good module is a
simple $\uqpg$-module.

\subsection{Generalized quantum affine Schur-Weyl duality functors} \label{sec:SWfunctor}
In this subsection, we recall the  construction of the generalized
quantum affine Schur-Weyl duality functor (\cite{KKK13A}).

Let $\uqpg$ be a quantum affine algebra over $\cor$ and let
$\{V_s\}_{s\in \mathcal{S}}$ be a family of good $\uqpg$-modules.
For each $s \in \Ss$, let $\lambda_s$ be a dominant extremal weight
of $V_s$ and let $v_s$ be a dominant extremal weight vector in $V_s$
of weight $\lambda_s$.

Assume that we have an index set $J$ and two maps
 $X \colon J \rightarrow \cor^\times$,
$S \colon J \rightarrow \Ss$.

For each $i$ and $j$ in $J$, we have a $\uqpg$-module homomorphism
\eqn \Rnorm_{V_{S(i)}, V_{S(j)}}(z_i,z_j) : (V_{S(i)})_\aff \tens
(V_{S(j)})_\aff \to \cor(z_i,z_j) \tens_{\cor[z^{\pm 1}_i,z^{\pm
1}_j]}(V_{S(j)})_\aff \tens (V_{S(i)})_\aff \eneqn which sends
$v_{S(i)}\tens v_{S(j)}$ to $v_{S(j)} \tens v_{S(i)}$. Here,
$z_i\seteq z_{V_{S(i)}}$ denotes the $\uqpg$-module automorphism on
$(V_{S(i)})_\aff$ of weight $\delta$. We denote by
$d_{V_{S(i)},V_{S(j)}}(z_j / z_i)$  the denominator of
$\Rnorm_{V_{S(i)}, V_{S(j)}}(z_i,z_j)$, which is the monic
polynomial in $z_j/z_i$ of the smallest degree such that \eqn
d_{V_{S(i)},V_{S(j)}}(z_j / z_i)\Rnorm_{V_{S(i)}, V_{S(j)}}(z_i,z_j)
\bl (V_{S(i)})_\aff \tens (V_{S(j)})_\aff \br \subset
(V_{S(j)})_\aff \tens (V_{S(i)})_\aff. \eneqn

We define a quiver $\Gamma^J$ associated with the datum $(J, X, S)$
 as follows:
\eq&&
\parbox{70ex}{\be[{(1)}]
\item we take $J$ as the set of vertices,
\item  we put $d_{ij}$ many arrows from $i$ to $j$,
where $d_{ij}$ denotes the order of the zero of
$d_{V_{S(i)},V_{S(j)}}(z_j / z_i)$ at $z_j / z_i = {X(j) / X(i)}$.
\ee}\label{gammaJ} \eneq

We define a symmetric Cartan matrix $A^J =(a^J_{ij})_{i,j\in J}$  by
\eq \label{eq:Cartan matrix} a^J_{ij}=\begin{cases}
2&\text{if $i=j$,}\\
-d_{ij}-d_{ji}&\text{if $i\not=j$.}\end{cases} \eneq

Set
\begin{equation}
P_{ij}(u,v) =(u-v)^{d_{ij}} c_{ij}(u,v),\label{def:Pij}
\end{equation}
where  $\{c_{ij}(u, v)\}_{i,j\in J}$  is a family of functions in
$\cor[[u,v]]$ satisfying \eq \label{cond:cij} c_{ii}(u,v)=1, \quad
c_{ij}(u,v) c_{ji}(v,u) =1. \eneq

Let $\set{\alpha_i}{i \in J}$ be the  set of simple roots
corresponding to the Cartan matrix $A^J$ and $\rootl^+_J=\sum_{i \in
J} \Z_{\ge 0}  \alpha_i$ be the corresponding positive root lattice.

Let $R^{J}( \beta )$  $(\beta \in \rootl^+_J)$ be the \KLR\
associated with the Cartan matrix $A^J$ and the  parameter
\begin{equation}\label{eq:Q_omega}
Q_{ij}(u,v) = \delta(i\not=j)P_{ij}(u,v)P_{ji}(v,u)=\delta(i\not=j)
(u-v)^{d_{ij}}(v-u)^{d_{ji}} \quad (i,j \in J).
\end{equation}

For each $\nu =(\nu_1,\ldots, \nu_n) \in J^{ \beta }$, let
\begin{equation*}
\Oh_{\mathbb{T}^n, X(\nu)} = \cor [[X_1 - X(\nu_1), \ldots,
X_n-X(\nu_n)]]
\end{equation*}
 be the completion of
the local ring $\sho_{\mathbb{T}^n, X(\nu)}$
 of $\mathbb{T}^n$ at  $X(\nu)\seteq(X(\nu_1),\ldots,X(\nu_n))$.
 Set
$$
V_\nu =( V_{S(\nu_1)})_\aff \otimes \cdots \otimes
(V_{S(\nu_n)})_\aff.$$ Then $V_{\nu}$ is a
$\bl\cor[X_1^{\pm1},\ldots, X_n^{\pm1}]\otimes\uqpg \br$-module,
where $X_k=z_{V_{S(\nu_k)}}$.
 We   define
\begin{align}
\hV_\nu \seteq \Oh_{\mathbb{T}^n, X(\nu)}\otimes
_{\cor[X_1^{\pm1},\ldots,X_n^{\pm1}]} V_\nu, \qquad \bV[ \beta ]
\seteq \soplus\nolimits_{\nu \in J^{ \beta }} \hV_\nu e(\nu).
\label{eq:Vhat}
\end{align}

The following theorem is one of the main result of \cite{KKK13A}.
\begin{theorem}
The space $\bV[\beta]$ is a $(\uqpg, R^J(\beta))$-bimodule.
\end{theorem}

For each $\beta \in \rootl^+_J$, we construct the functor
\begin{align}
\F_{ \beta } \colon \Mod(R^J(\beta)) &\rightarrow \Mod(\uqpg)
\end{align}
defined by \eq \F_{ \beta } (M) \seteq \bV[{ \beta }] \otimes_{R^J({
\beta })} M, \label{eq:the functor}\eneq where $M$ is an
$R^J(\beta)$-module.

Set \eqn \F\seteq \soplus_{\beta\in \rootl^+_J}\F_\beta
\col\soplus_{\beta\in \rootl^+_J} \Mod(R^J(\beta)) \rightarrow
\Mod(\uqpg). \eneqn

\begin{theorem}[{\cite{KKK13A}}] \label{thm:exact}
If the Cartan matrix $A^J$ associated with  $R^J$  is of finite type
$A, D$ or $E$, then the functor  $\F_\beta$ is exact  for every
$\beta \in \rootl^+_J$.
\end{theorem}

For each $i \in J$, let $L(i)$ be the $1$-dimensional
$R^J(\alpha_i)$-module  generated by a nonzero vector $u(i)$ with
relation $x_1 u(i) = 0$ and $e(j) u(i)=\delta(j=i) u(i)$ for $j \in
J$. The space $L(i)_z :=  \cor[z] \otimes L(i)$ admits an
$R^J(\alpha_i)$-module structure as follows: \eqn x_1(a \otimes
u(i)) =(za) \otimes u(i), \quad e(j)(a\otimes u(i)) = \delta_{j,i}(a
\otimes u(i)). \eneqn Note that it is isomorphic to  $R^J(\alpha_i)$
as a left $R^J(\alpha_i)$-module.

By the construction in \cite{KKK13A}, we have
\begin{prop} \label{prop:image of tau} \hfill

\bnum
\item For any $i \in J$,
we have
\begin{align}
  \mathcal F(L(i)_z) \simeq \cor[[z]]\tens_{\cor[z_{V_{S(i)}}^{\pm 1}]}(V_{S(i)})_\aff,
\end{align}
where   $\cor[z_{V_{S(i)}}^{\pm 1}]\to\cor[[z]]$ is  given  by
$z_{V_{S(i)}}\mapsto X(i)(1+z)$.

\item For $i,j\in J$,
let   $$\phi=R_{L(i)_z,L(j)_{z'}}\col L(i)_z \conv L(j)_{z'}
\rightarrow L(j)_{z'} \conv L(i)_z.$$ That is, let $\phi$ be  the
  $R^J(\alpha_i+\alpha_j)$-module homomorphism given by
\begin{align}
  \phi\bl u(i)_z\otimes u(j)_{z'}\br = \vphi_1\bl u(j)_{z'} \otimes u(i)_z\br,
\end{align}
where $\vphi_1$ is the intertwiner in \eqref{def:int}. Then we have
\eqn &&\mathcal F(\phi)
=(X_i/X(i)-X_j/X(j))^{d_{i,j}}c_{i,j}(X_i/X(i)-1,X_j/X(j)-1)
\Rnorm_{V_{S(i)}, V_{S(j)}}\eneqn as a $U_q'(\g)$-module
homomorphism \eqn &&\Oh_{\mathbb T^2,( X(i),X(j))
}\tens_{\cor[X_i^{\pm 1},X_j^{\pm 1}]}
\bl(V_{S(i)})_\aff \otimes (V_{S(j)})_\aff\br \\
&&\hs{10ex}\longrightarrow\Oh_{\mathbb T^2,( X(j),X(i)) }
\tens_{\cor[X_j^{\pm 1},X_i^{\pm 1}]} \bl(V_{S(j)})_\aff \otimes
(V_{S(i)})_\aff\br, \eneqn where ${X_i}=z_{V(S(i))}$ and
${X_j}=z_{V(S(j))}$. \ee
\end{prop}

  Recall that $\CC_\g$ denotes the category of
finite-dimensional integrable $\uqpg$-modules.

\begin{theorem} \label{thm:conv to tensor}
The functor $\F$ induces a tensor functor
$$\F\col\soplus_{\beta \in \rootl^+_J} R^J( \beta) \gmod
\to\CC_\g.$$ Namely, $\F$ sends finite-dimensional graded
$R^J(\beta)$-modules to $\uqpg$-modules in $\CC_\g$, and  there
exist canonical $\uqpg$-module isomorphisms \eqn \F(R^J(0))\simeq
\cor,  &  \F(M_1 \conv M_2) \simeq \F(M_1) \otimes \F(M_2) \eneqn
for $M_1 \in R^J( \beta_1 ) \gmod$ and $M_2 \in R^J( \beta_2 )\gmod$
 such that the diagrams in  \cite[A.1.2]{KKK13A} is commutative .
\end{theorem}

\section{Comparison of denominators in untwisted cases and twisted cases} \label{sec:comparison}
\subsection{Denominators of normalized $R$-matrices for $U_q'(A^{(1)}_{N-1})$ and $U_q'(A^{(2)}_{N-1})$}
In this subsection, we recall the denominators of normalized
$R$-matrices for quantum affine algebras of type A. In Table
\ref{table:Dynkin}, we list the Dynkin diagrams with an enumeration
of vertices by simple roots and list the corresponding fundamental
weights.

\begin{table}[ht]
\renewcommand{\arraystretch}{2.2}
\centering
\begin{tabular}[c]{c|c|p{5cm}}
Type& Dynkin diagram & ${}^{\quad}$  Fundamental weights\\
\hline \hline
 $A_{n}^{(1)}$
($n \ge 2$) &
$$
\xymatrix@R=3ex{ & &&*{\circ}<3pt> \ar@{-}[drrr]^<{\alpha_0} \ar@{-}[dlll]\\
*{\circ}<3pt> \ar@{-}[r]_<{\alpha_1}  &*{\circ}<3pt>
\ar@{-}[r]_<{\alpha_2} & *{\circ}<3pt> \ar@{-}[r]_<{\alpha_3} & {}
\ar@{.}[r]  & *{\circ}<3pt> \ar@{-}[r]_>{\,\,\,\ \alpha_{n-1}}
&*{\circ}<3pt>\ar@{-}[r]_>{\,\,\,\,\alpha_n} &*{\circ}<3pt> }
$$ & $\varpi_i = \Lambda_i -\Lambda_0$ $(1 \le i \le n)$\\
\hline $A_{2}^{(2)}$ & \hskip-3em $$ \put(0,0){\circle{4}}
\put(40,0){\circle{4}} \put(3.7,-3.6){\line(1,0){23}}
\put(5,-1.15){\line(1,0){26}} \put(5,1.15){\line(1,0){26}}
\put(3.7,3.6){\line(1,0){23}} \put(20.5,-5.2){{\LARGE\mbox{$>$}}}
\put(-3,-10){\tiny $\alpha_0$} \put(37,-10){\tiny $\alpha_1$}
$$
& $\varpi_1=2\Lambda_1- \Lambda_0$
 \\
\hline $A_{3}^{(2)}$ &
$$\xymatrix@R=3ex{*{\circ}<3pt>  \ar@{<=}[r]_<{\alpha_0} & *{\circ}<3pt> \ar@{=>}[r]_<{\alpha_2} \ar@{}[r]_>{\,\,\,\ \alpha_1}  & *{\circ}<3pt>}$$
 & $\varpi_1=\Lambda_1- \Lambda_0$,  \newline   $\varpi_2=\Lambda_2-2\Lambda_0$ \\
\hline $A_{2n-1}^{(2)}$ ($n \ge 3$) &
$$
\xymatrix@R=3ex{ & *{\circ}<3pt> \ar@{-}[d]^<{\alpha_0}\\
*{\circ}<3pt> \ar@{-}[r]_<{\alpha_1}  &*{\circ}<3pt>
\ar@{-}[r]_<{\alpha_2}  &
  *{\circ}<3pt> \ar@{-}[r]_<{\alpha_3} & {}
\ar@{.}[r] & *{\circ}<3pt> \ar@{-}[r]_>{\,\,\,\ \alpha_{n-1}}
&*{\circ}<3pt>\ar@{<=}[r]_>{\,\,\,\,\alpha_{n}} &*{\circ}<3pt> }
$$ & $\varpi_i=\Lambda_i-\Lambda_0$ \newline $(i=1,n)$,  \newline $\varpi_i=\Lambda_i- 2\Lambda_0$ \newline $(2\le i \le n-1)$  \\
\hline $A_{2n}^{(2)}$ ($n \ge 2$) & $$ \xymatrix@R=3ex{
*{\circ}<3pt> \ar@{=>}[r]_<{\alpha_0}  &*{\circ}<3pt>
\ar@{-}[r]_<{\alpha_1}   & *{\circ}<3pt> \ar@{-}[r]_<{\alpha_2} & {}
\ar@{.}[r] & *{\circ}<3pt> \ar@{-}[r]_>{\,\,\,\ \alpha_{n-1}}
&*{\circ}<3pt>\ar@{=>}[r]_>{\,\,\,\,\alpha_n} &*{\circ}<3pt> }
$$ & $\varpi_i = \Lambda_i -\Lambda _0$
\newline $(i=1,\ldots,n-1)$, \newline $\varpi_n =2\Lambda_n - \Lambda_0 \ $ \\ \hline
\end{tabular}

\caption{Dynkin diagrams and fundamental weights}
\label{table:Dynkin}
\end{table}

%\begin{remark}
% Even though our enumeration of vertices of Dynkin diagram of type $A^{(2)}_{2n}$ is different from the one in \cite{Kac},
%for each $i=1,\ldots,n$ the corresponding fundamental
%representations $V(\varpi_i)$ are isomorphic to each other, since
%the corresponding fundamental weights are conjugate to each other,
%%under the Weyl group action (see \cite[\S 5.2]{Kas02}).  In
%Table\,\ref{table:Dynkin} and in the sequel, we denote by
%$A^{(2)}_3$ the Dynkin diagram of type  $D^{(2)}_3$ in \cite{Kac}.
%\end{remark}

By \cite{DO94}, for $\g=A^{(1)}_{N-1}$ ($N \ge 1$), $1\leq k,l \leq
N-1$,  we have \eq\hs{10ex} \label{eq:denom A^(1)_n} d_{V(\varpi_k),
V(\varpi_{\ell})}(z)=\hs{-4ex} \prod\limits_{s=1}^{\min(k, \ell,
N-k, N-\ell)}\hs{-2ex} \bl z-(-q)^{|k-\ell|+2s}\br. \label{eq:zerR}
\eneq

We recall the denominators of normalized $R$-matrices between
fundamental representations of type $A^{(2)}_{N-1}$, given in
\cite{Oh14}.
  \begin{theorem}
For $\g=A^{(2)}_{N-1}$ $(N \ge 3)$, $1\leq k,l \leq \lfloor N/2
\rfloor$, we have
\begin{align} \label{eq:denom A^(2)_N-1}
  d_{k,l}(z) = \prod_{s=1}^{\min(k,l)} (z-(-q)^{|k-l|+2s})(z+q^N(-q)^{-k-l+2s}).
\end{align}
\end{theorem}

\begin{remark}
 Even though our enumeration of vertices of Dynkin diagram of type $A^{(2)}_{2n}$ is different from the one in \cite{Kac},
for each $i=1,\ldots,n$ the corresponding fundamental
representations $V(\varpi_i)$ are isomorphic to each other, since
the corresponding fundamental weights are conjugate to each other
under the Weyl group action (see \cite[\S 5.2]{Kas02}). For the sake
of notational simplicity, we denote the Dynkin diagram of type
$D^{(2)}_3$ in \cite{Kac} by  $A^{(2)}_3$ throughout this paper.
\end{remark}

\subsection{The quiver isomorphism}
For each quantum affine algebra $U_q(\g)$, we define a quiver
$\mathscr S(\g)$ as follows: \eq&&
\parbox{70ex}{\be[{(1)}]
\item we take  the set of equivalence classes
$\hat I_{\g} \seteq (I_0 \times \cor^\times) /  \sim$ as the set of
vertices, where the equivalence relation is given by $(i,x) \sim
(j,y)$
 if and only if $V(\varpi_i)_x \cong V(\varpi_j)_y$,

\item  we put $d$ many arrows from $(i,x)$ to $(j,y)$,
where $d$ denotes the order of the zero of
$d_{V(\varpi_i),V(\varpi_j)}(z_{V(\varpi_j)} / z_{V(\varpi_i)})$ at
$z_{V(\varpi_j)} / z_{V(\varpi_i)} = {y / x}$. \ee} \eneq Note that
$(i,x)$ and $(j,y)$ are linked by at least one arrow in $\mathscr
S(\g)$  if and only if the tensor product $V(\varpi_i)_x \tensor
V(\varpi_j)_y$ is reducible (\cite[Corollary 2.4]{AK}).

Let $\mathscr S_0(\g)$ be a connected component of $\mathscr S(\g)$.
Note that a connected component of $\mathscr S(\g)$ is unique up to
a spectral parameter and hence $\mathscr S_0(\g)$ is uniquely
determined up to a quiver isomorphism. For example, one can take \eq
&&\mathscr S_0(A^{(1)}_{n}) \seteq \{(i,(-q)^{p}) \in \{1,\ldots,n\} \times \cor^{\times} \,;\, p \equiv i+1 \hs{-1ex}\mod \ 2 \}, \label{eq:S0A^(1)_n} \\
&&\mathscr S_0(A^{(2)}_{2n-1}) \seteq \{( i,\pm(-q)^{p})
\in \{1,\ldots,n\} \times \cor^{\times} \,;\,  \label{eq:S0A^(2)_2n-1} \\
&&\hs{35ex}\text{$i  \in \{1,\ldots,n\}$, $p \equiv i+1 \hs{-1ex}\mod  2$} \}, \nonumber  \\
&&\mathscr S_0(A^{(2)}_{2n}) \seteq \{(i,(-q)^{p}) \in
\{1,\ldots,n\} \times \cor^{\times} \,;\, p \in \Z \}.
\label{eq:S0A^(2)_2n} \eneq
Remark that we have $V(\varpi_n)_x\simeq V(\varpi_n)_{-x}$ in the
$A^{(2)}_{2n-1}$-case.

Let $\CC^0_\g$ be the smallest full subcategory of $\CC_\g$ stable
under taking subquotients, extensions, tensor products and
containing $\set{V(\varpi_i)_x}{(i,x) \in \mathscr S_0(\g)}$.

Define a map \eqn \pi^{(2)}_{N-1} : \hat I_{A^{(1)}_{N-1}} \To \hat
I_{A^{(2)}_{N-1}} \eneqn by \eq \label{eq: pi2 map}
\pi^{(2)}_{N-1}(i,x)=
\begin{cases}
(i, x) & \text{if} \   1 \le i \le \lfloor N/2 \rfloor,  \\
(N-i, (-1)^{N-1} x) & \text{if} \   \lfloor N/2 \rfloor < i \le N-1.  \\
\end{cases}
\eneq When there is no afraid of confusion, then we just write
$\pi^{(2)}$ instead of $\pi^{(2)}_{N-1}$.

By \eqref{eq:denom A^(1)_n} and \eqref{eq:denom A^(2)_N-1},
 we have the following

\begin{prop} \label{prop:quiver isom}
The map $\pi^{(2)}$ induces quiver isomorphisms
$$\mathscr S({A^{(1)}_{N-1}}) \overset{\sim} \longrightarrow \mathscr S({A^{(2)}_{N-1}}) \ \
\text{and} \ \ \mathscr S_0({A^{(1)}_{N-1}}) \overset{\sim}
\longrightarrow \mathscr S_0({A^{(2)}_{N-1}}).$$
\end{prop}

To avoid the confusion, we use the notation $V^{(t)}(\varpi_i)$ for
the fundamental representation of weight $\varpi_i$ of
$U_q'(A^{(t)}_{N-1})$ ($t=1,2$).
 We  also  use  the following notation:
for $(i,x) \in \mathscr S (A^{(t)}_{N-1})$, set $V^{(t)}(i,x)\seteq
V^{(t)}(\varpi_i)_x$ ($t=1,2$). We write $V(i,x)$ instead of
$V^{(t)}(i,x)$ when there is no afraid of confusion.

We record the following propositions here for the later use.

\begin{prop} [{\cite[Theorem 4.15]{Her101}}] \label{prop:dim} For all
$(i,x) \in \mathscr S({A^{(1)}_{N-1}})$, we have
$$\dim_\cor V^{(1)}(i,x) = \dim_\cor V^{(2)}(\pi^{(2)}(i,x)).$$

\end{prop}

\begin{prop}
[{\cite[Theorem 6.1]{CP96}}] \label{prop:CP} Let $\g=A^{(1)}_{N-1}$,
$1 \le i,j,k \le N-1$, $x,y,z\in \cor^\times$. Then \eqn
\Hom_{U'_q(\g)}(V^{(1)}(\varpi_i)_x \otimes V^{(1)}(\varpi_j)_y,
V^{(1)}(\varpi_k)_z) \neq 0 \eneqn if and only if one of the
following conditions holds: \bnum
\item
$ i+j < N$, $k=i+j$,  $x/z= (-q)^{-j}$, $ y/z = (-q)^i$,
\item
$i+j > N$, $ k= i+j-N$, $ x/z=(-q)^{-N+j}$,  $ y/z=(-q)^{N-i}$.
\end{enumerate}
\end{prop}

\begin{prop}[{ \cite[Theorem 3.5, Theorem 3.9]{Oh14}}] \label{prop:Oh14}
For $1 \le i, j \le \lfloor N/2 \rfloor $ such that $i+j \le \lfloor
N/2 \rfloor$,
 there exists an exact sequence
\eqn&& 0\To V^{(2)}(\varpi_{i+j})\To[\iota_{i,j}]
V^{(2)}(\varpi_j)_{(-q)^{i}}
\otimes V^{(2)}(\varpi_i)_{(-q)^{-j}}\\
&&\hs{20ex}\To[\ h\ ] V^{(2)}(\varpi_i)_{(-q)^{-j}} \otimes
V^{(2)}(\varpi_j)_{(-q)^{i}}
 \To[p_{i,j}] V^{(2)}(\varpi_{i+j})\To0.
 \eneqn
  Assume that $N-1=2n$ $(n \ge 2)$.
  Then there exists an exact sequence
of $U_q'(A^{(2)}_{2n})$-modules \eqn &&0\To V^{(2)}(\varpi_{n})
\To[\iota_{n,1}] V^{(2)}(\varpi_1)_{(-q)^{n}}
  \otimes V^{(2)}(\varpi_n)_{(-q)^{-1}}\\
  &&\hs{20ex}\To[\ \ ] V^{(2)}(\varpi_n)_{(-q)^{-1}} \otimes V^{(2)}(\varpi_1)_{(-q)^{n}}
  \To[p_{n,1}] V^{(2)}(\varpi_{n}) \To0.
  \eneqn
 \end{prop}

\section{The functor $\F^{(2)}$}\label{sec:F2}
\subsection{Symmetric quiver Hecke algebra of type $A_\infty$}
Let $V=V^{(2)}(\varpi_1)$ be the fundamental representation of
$U_q'(A^{(2)}_{N-1})$ with extremal weight $\varpi_1$. Let
$\mathcal{S}=\{V\}$, $J=\Z$ and let $X\col J\to \cor^\times$ be the
map given by $X(j)=q^{2j}$. Then we have \eq
d_{ij}=\delta(j=i+1)\qtext{for $i,j\in J$.} \eneq For $i,j\in J$, we
have
\begin{align*}(\al_i,\al_j)=\begin{cases}-1&\text{if $i-j=\pm1$,}\\
2&\text{if $i=j$,}\\
0&\text{otherwise,}
\end{cases} \\
P_{ij}(u,v)=(u-v)^{\delta(j=i+1)}c_{i,j}(u,v),\end{align*} and
$$Q_{ij}(u,v)=\begin{cases}
\pm(u-v)&\text{if $j=i\pm1$,}\\
0&\text{if $i=j$,}\\
1&\text{otherwise.}
\end{cases}$$
The family $\{c_{i,j}(u,v)\}_{i,j\in J}$ will be given later in
\eqref{eq:c}.

Therefore the corresponding \KLR\ $R$ is of type
$\mathrm{A}_{\infty}$.
\medskip

We take
$$P_J=\soplus_{a\in \Z}\Z\oep_a$$
as the weight lattice with $(\oep_a,\oep_b)=\delta_{a,b}$. The root
lattice $\rtl_J=\soplus_{i\in J}\Z\al_i$ is embedded into $P_J$ by
$\al_i=\oep_i-\oep_{i+1}$. We write  $\rtl_J^+$ for $\soplus_{i\in
J}\Z_{\ge0}\al_i$.

Recall  that
 the functor
  $$\F^{(2)} \col \soplus_{\beta\in  \rtl^+_J} \Modg(R(\beta))\to
\Mod( \UAtwo) $$
 defined in \eqref{eq:the functor} is exact (Theorem \ref{thm:exact}).

\bigskip

\subsection{Segments}
 A pair of integers $(a,b)$ such that $ a \leq b$
is called a {\em segment}.  The \emph{length} of $(a,b)$  is
$b-a+1$. A \emph{multisegment} is a  finite sequence of  segments.

 For a segment $(a,b)$ of length $\ell$, we define a graded $1$-dimensional
 $R( \oep_a-\oep_{b+1} )$-module
 $L(a,b)=\cor u{(a,b)}$ in $R( \oep_a-\oep_{b+1} )\gmod$ which is generated by a
vector $u{(a,b)}$ of degree $0$ with the action of $R(
\oep_a-\oep_{b+1} )$ given by
 \begin{align}
x_m u{(a,b)} =0 , && \tau_k u{(a,b)} =0 , && e(\nu) u{(a,b)}=
\begin{cases}
u{(a,b)} & \text{if $\nu=(a,a+1, \ldots, b)$,} \\
0 & \text{otherwise.}
\end{cases}
\end{align}
We understand that $L(a,a-1)$ is the 1-dimensional module over $R(0)
= \cor$  and the length of $(a,a-1)$ is $0$. When $a=b$, we use the
notation $L(a)$ instead of $L(a,a)$.

We give a total order on the set of segments as follows: \eqn
\label{eq:right order} (a_1,b_1) > (a_2,b_2) \quad  \ \text{if} \
a_1 > a_2 \ \text{or} \ a_1=a_2 \ \text{and} \ b_1 > b_2. \eneqn

Then we have
\begin{prop} [{\cite[Theorem 4.8, Theorem 5.1]{KP11},
\cite[Proposition 4.2.7]{KKK13A}}] \label{prop:multisegments}\hfill
\bnum
\item
Let $M$ be a finite-dimensional simple  graded  $R(\ell)$-module.
Then there exists a unique pair of a multisegment $\big ((a_1,b_1),
\ldots , (a_t,b_t) \big)$ and an integer $c$ such that \bna
\item $(a_k, b_k) \ge (a_{k+1}, b_{k+1})$ for $1 \leq k \leq t-1$,
\item $\sum_{k=1}^t \ell_k=\ell$, where $\ell_k \seteq b_k -a_k +1$,
\item
$ M \simeq q^c\hd \big(L(a_1,b_1)\conv \cdots \conv L(a_t,b_t)
\big)$,
 where $\hd$ denotes the head.
\ee

\item Conversely, if  a multisegment   $\big ((a_1,b_1), \ldots , (a_t,b_t) \big)$  satisfies
$\rm (a)$ and $\rm (b)$, then $\hd \big(L(a_1,b_1)\conv \cdots \conv
L(a_t,b_t) \big)$ is a simple   graded  $R(\ell)$-module. Moreover,
$\hd \big(L(a_1,b_1)\conv \cdots \conv L(a_t,b_t) \big)$ is
isomorphic to $\Im \bl \rmat{L(a_1,b_1), \cdots, L(a_t,b_t)} \br$ up
to a grade shift. \ee
\end{prop}

If a multisegment $\big ((a_1,b_1), \ldots , (a_t,b_t) \big)$
satisfies the condition (a) above, then we say that it is an {\em
ordered multisegment}. We call the  ordered  multisegment $\big
((a_k,b_k)\big)_{1\le k \le t}$ in Proposition
\ref{prop:multisegments}~(i) the {\em multisegment associated with
$M$}.

\begin{prop} [{\cite[Proposition 4.2.3]{KKK13A} }]\label{prop:exact sequences}
For $a \leq b$ and $a' \leq b'$,
 set $\ell=b-a+1$, $\ell'=b'-a'+1$, $\beta=\oep_{a}-\oep_{b+1}$
and $\beta'=\oep_{a'}-\oep_{b'+1}$.
  \bnum
\item If $a'=a$ and $b'=b$, then
we have $\rmat{L(a,b),L(a,b)}=\id_{L(a,b)\conv  L(a,b)}$.
  \item
\bna
\item
If $a \leq a' \leq b \leq b'$, then
 there exists a nonzero homomorphism
  \begin{align*}
  \rmat{L(a,b),L(a',b')} : L(a,b) \conv L(a',b') \to q^{\delta_{a,a'}+\delta_{b,b'}-2}L(a',b') \conv L(a,b).
  \end{align*}
\item  Unless $a \leq a' \leq b \leq b'$,
there exists a nonzero homomorphism
 \begin{align*}
   g\seteq \rmat{L(a,b),L(a',b')}: L(a,b) \conv L(a',b')\to q^{(\beta,\beta')}
 L(a',b') \conv L(a,b).
  \end{align*}
\ee

\item If $a\le a'\le b'\le b$, then
 $L(a,b) \conv L(a',b')$ is simple and
$$L(a,b) \conv L(a',b') \simeq q^{\delta_{a,a'}-\delta_{b,b'}}L(a',b') \conv L(a,b).$$
\item If $b'<a-1$,
then  $L(a,b) \conv L(a',b')$ is simple and
$$g\col L(a,b) \conv L(a',b')\isoto L(a',b') \conv L(a,b).$$
  \item If $a' < a\le b' < b$, then
we have the following exact sequence
\begin{align*}
 & 0 \To qL(a',b) \conv L(a, b') \To
  L(a, b) \conv L(a',b')\\
&\hs{20ex}\To[{\hs{1.5ex} g\hs{1.5ex}}] L(a',b') \conv L(a, b)\To
 q^{-1}L(a', b) \conv L(a,b')\To 0.\end{align*}
Moreover, the image of $g$ coincides with the head of $ L(a, b)\conv
L(a',b')$ and the socle of $ L(a',b') \conv L(a, b)$.
\item If $a=b'+1$, then we have an exact sequence
\begin{align*}
  0 \to q L(a',b) \To
 L(a, b) \conv L(a',b')
\To[{g}] q^{-1}L(a',b') \conv L(a, b)\to q^{-1} L(a', b)\rightarrow
0.
\end{align*}
Moreover, the image of ${g}$ coincides with the head of $ L(a,
b)\conv L(a',b')$ and the socle of $ q^{-1}L(a',b') \conv L(a, b)$.
\ee
\end{prop}

\subsection{Properties of the functor $\F^{(2)}$}

For $k > \lfloor N /2 \rfloor +1 $ or $k < 0$, $V^{(2)}(\varpi_{k})$
is understood to be zero, and the modules $V^{(2)}(\varpi_{0})$ and
$V^{(2)}(\varpi_{\lfloor N /2 \rfloor+1}) $ are understood to be the
trivial representation.

\begin{prop} \label{prop:image of fund. repns.}
Let $(a,b)$ be a segment with length $\ell\seteq b-a+1$.
 Then we have
\begin{align*}
\F^{(2)}(L(a,b)) \simeq V(\pi^{(2)}(\ell, (-q)^{a+b})).
\end{align*}
\end{prop}

\begin{proof}
We will show our assertion by induction on $\ell$. In the course of
the proof, we {\it omit} the grading of modules over quiver Hecke
algebras. When $\ell =1$, we have $\F^{(2)}(L(a)) \simeq
V_{(-q)^{2a}}$ by Proposition \ref{prop:image of tau} (i).

\noindent Assume that $2 \le \ell \le N$. Consider the following
exact sequence in $R(\ell) \smod$
\begin{align} \label{eq:exact sequnce of segments}
  0 \rightarrow L(a,b) \rightarrow L(b) \conv L(a,b-1)
  \To[{\rmat{(b),(a,b-1)}}] L(a,b-1) \conv L(b)
 \rightarrow L(a,b) \rightarrow 0
\end{align}
given in Proposition \ref{prop:exact sequences} (vi). Here, we write
\eqn && \rmat{(a,b),(a',b')}\col L(a,b)\conv L(a',b')\To
L(a',b')\conv L(a,b) \eneqn for $\rmat{L(a,b),L(a',b')}$. Applying
the exact functor $\F^{(2)}$ and using the induction hypothesis, we
obtain an exact sequence \eq &&\ba{l}
0\to \F^{(2)}(L(a,b))\to V_{(-q)^{2b}}\tens V(\pi^{(2)}(\ell-1,(-q)^{a+b-1}))\\[2ex]
\hs{7ex}\To[{\;\F^{(2)}(\rmat{(b),(a,b-1)})\;}]
V(\pi^{(2)}(\ell-1,(-q)^{a+b-1})) \tens V_{(-q)^{2b}
}\to\F^{(2)}(L(a,b))\to0.\ea \eneq
If  $\F^{(2)}(\rmat{(b),(a,b-1)})$ vanishes, then we have
$$V_{(-q)^{2b}}\tens V(\pi^{(2)}(\ell-1, (-q)^{a+b-1}))
\simeq V(\pi^{(2)}(\ell-1, (-q)^{a+b-1}))\tens V_{(-q)^{2b}}$$
because they are both isomorphic to $\F^{(2)}(L(a,b))$. Hence
$V_{(-q)^{2b}}\tens V(\pi^{(2)}(\ell-1, (-q)^{a+b-1}))$ is simple,
which is a contradiction. Therefore, we have
$\F^{(2)}(\rmat{(b),(a,b-1)}) \neq 0$.

On the other hand, by Proposition \ref{prop:Oh14} we have an exact
sequence \eq &&\ba{l} 0\to V(\pi^{(2)}(\ell, (-q)^{a+b}))
\to V_{{(-q)}^{2b}}\tens V(\pi^{(2)}(\ell-1, (-q)^{a+b-1}))\\[2ex]
\hs{10ex}\To[\;h\;]V(\pi^{(2)}(\ell-1, (-q)^{a+b-1}))\tens
V_{(-q)^{2b} }\to  V(\pi^{(2)}(\ell, (-q)^{a+b}))\to0\ea \eneq such
that $h$ is non-zero. Here $ V(\pi^{(2)}(N, (-q)^{a+b}))$ is
understood to be the trivial representation. Since $2b > a+b-1$,
\cite[Theorem 2.2.1]{KKK13A} implies that the module
$V_{(-q)^{2b}}\tens V(\pi^{(2)}(\ell-1, (-q)^{a+b-1}))$ is generated
by the tensor product
 of dominant extremal weight vectors.
Hence we obtain
$$\Hom_{\UAtwo} ( V_{(-q)^{2b}}\tens V(\pi^{(2)}(\ell-1, (-q)^{a+b-1})),
V(\pi^{(2)}(\ell-1, (-q)^{a+b-1}))\tens V_{(-q)^{2b}})=\cor h.$$
Thus $\F^{(2)}(\rmat{(b),(a,b-1)})$  is equal to $h$ up to a
constant multiple and hence $\F^{(2)}(L(a,b))$ is isomorphic to $
V(\pi^{(2)}(\ell, (-q)^{a+b}))$. Thus we have proved the proposition
when $ \ell  \le N$.

Now assume that $\ell=N+1$. Then
$\F^{(2)}(L(a,b-1))\simeq\F^{(2)}(L(a-1,b))\simeq \cor$. Applying
$\F^{(2)}$ to  the  epimorphism $L(a,b-1)\conv L(b)\epito L(a,b)$,
$\F(L(a,b))$ is a quotient of $V_{(-q)^{2b}}$. Similarly, applying
$\F^{(2)}$ to   the   epimorphism $L(a)\conv L(a+1,b)\epito L(a,b)$,
$\F(L(a,b))$ is a quotient of $V_{(-q)^{2a}}$. Since $V_{(-q)^{2b}}$
and $V_{(-q)^{2a}}$ are simple modules and they are not isomorphic
to each other, we conclude that $\F^{(2)}(L(a,b))$ vanishes.

For $\ell>N+1$, $\F^{(2)}(L(a,b))$ vanishes since it is a quotient
of
$$\F^{(2)}(L(a,a+N))\tens\F^{(2)}(L(a+N+1,b))\simeq 0 \tens\F^{(2)}(L(a+N+1,b))\simeq0.$$
\end{proof}

\begin{corollary} \label{cor:surjhoms}
If one of conditions {\rm(i)}, {\rm(ii)} in {\rm Proposition
\ref{prop:CP}} holds, then we have \eqn
\Hom_{\UAtwo}\big(V(\pi^{(2)}(i,x)) \otimes V(\pi^{(2)}(j,y)),
V(\pi^{(2)}(k,z)) \big) \neq 0. \eneqn
\end{corollary}

\begin{proof} Note that in the both cases, $(i,x)$ and $(j,y)$  are linked by an arrow
in $\mathscr S(A^{(2)}_{N-1})$. Thus we have \eq \label{eq:non-isom}
V(\pi^{(2)}(i,x)) \otimes V(\pi^{(2)}(j,y)) \ncong V(\pi^{(2)}(j,y))
\otimes V(\pi^{(2)}(i,x)). \eneq Indeed, if these tensor products
were isomorphic, then by \cite[Corollary 3.9]{KKKO14} they would be
simple, which is a contradiction.

Assume that  condition (i) holds. We may assume that $x=(-q)^{1-i}$,
$y=(-q)^{j+1}$ and $z=(-q)^{-i+j+1}$. Take  the segments $(a,
b)\seteq(1, j)$ and $(a',b')=(1-i,0)$. Applying $\F^{(2)}$ to the
exact sequence in Proposition~\ref{prop:exact sequences} (vi), we
obtain a surjective homomorphism \eqn && V(\pi^{(2)}(i,
(-q)^{a'+b'}))\tens V(\pi^{(2)}(j, (-q)^{a+b})) \epito
V(\pi^{(2)}(i+j, (-q)^{a'+b})). \eneqn

Assume that condition (ii) holds. We may assume that $x=(-q)^{i+1}$,
$y=(-q)^{2N-j+1}$ and $z=(-q)^{N+i-j+1}$. Take  the segments $(a,
b)\seteq(N-j+1, N)$ and $(a',b')=(1,i)$. Applying $\F^{(2)}$ to the
exact sequence in Proposition~\ref{prop:exact sequences} (v),  we
obtain a surjective homomorphism \eqn &&
 V(\pi^{(2)}(i, (-q)^{a'+b'}))\tens V(\pi^{(2)}(j, (-q)^{a+b}))\\
&&\hs{10ex}\epito  V(\pi^{(2)}(N, (-q)^{a'+b}))\tens
V(\pi^{(2)}(i+j-N, (-q)^{a+b'})) . \eneqn Since $V(\pi^{(2)}(N,
(-q)^{a'+b}))$ is isomorphic to the trivial representation, we
obtain the desired result.
\end{proof}

\begin{lemma}\label{lem:intertwiners}
Assume that two segments $(a,b)$ and $(a',b')$ satisfy $(a,b)\ge
(a',b')$. Set $\ell=b-a+1$, $\ell'=b'-a'+1$, $x=(-q)^{a+b}$ and
$x'=(-q)^{a'+b'}$. Then the following statements hold. \bnum
\item
$x'/{x}$ is not a zero of the denominator
$d_{V^{(2)}(\varpi_\ell),V^{(2)}(\varpi_{\ell'})}(z'/z)$ of the
normalized $R$-matrix
$\Rnorm_{V^{(2)}(\varpi_\ell),V^{(2)}(\varpi_{\ell'})}(z,z')$.
\item
 The  homomorphism
  $$\F^{(2)}(\rmat{L(a,b),L(a',b')}) \col
 V(\pi^{(2)}(\ell,x)) \otimes V(\pi^{(2)}(\ell',x')) \rightarrow
V(\pi^{(2)}(\ell',x')) \otimes V(\pi^{(2)}(\ell,x)) $$
   is a non-zero constant multiple of the normalized $R$-matrix
  $\Rnorm_{V(\varpi_{\ell}), V(\varpi_{\ell'})} (x,x')$.
\ee
\end{lemma}

\begin{proof}
Because $(b'-a'+1)-(b-a+1)\ge (a'+b')-(a+b)$, there is no arrow in
$\mathscr S(A^{(1)}_{N-1})$ from $(\ell,x)$ to $(\ell', x')$, and
hence so is for $\pi^{(2)}(\ell,x)$ and $\pi^{(2)}(\ell', x')$  in
$\mathscr S(A^{(2)}_{N-1})$,  by Proposition \ref{prop:quiver isom}.
Thus we have (i).

By (i) and  \cite[Theorem 2.2.1]{KKK13A}, the module
$V(\pi^{(2)}(\ell,x)) \otimes V(\pi^{(2)}(\ell',x'))$ is generated
by the tensor product $v_\ell\tens v_{\ell'}$ of dominant extremal
weight vectors. Thus any non-zero homomorphism from
 $ V(\pi^{(2)}(\ell,x)) \otimes V(\pi^{(2)}(\ell',x'))$ to
  $V(\pi^{(2)}(\ell',x')) \otimes V(\pi^{(2)}(\ell,x)) $
is  a constant multiple of the normalized $R$-matrix.
 Hence it is enough to show that
$\F^{(2)}(\rmat{L(a,b),L(a',b')})$ does not vanish.

We may therefore assume that $r\seteq\rmat{L(a,b),L(a',b')}$ is not
an isomorphism. Equivalently, $(\ell,x)$ and $(\ell',x')$ are linked
by an arrow in $\mathscr S_0(A^{(1)}_{N-1})$. Since $\mathscr
S_0(A^{(2)}_{N-1})$ is
 isomorphic to $\mathscr S_0(A^{(1)}_{N-1})$ as a quiver,
$\pi^{(2)}(\ell,x)$ and $\pi^{(2)}(\ell',x')$ are linked. Hence,
$V(\pi^{(2)}(\ell,x)) \otimes V(\pi^{(2)}(\ell',x'))$ is not simple.

On the other hand, since $(\ell,x)$ and $(\ell',x')$ are linked, we
have $a'<a\le b'<b$   or $a=b'+1$. Applying $\F^{(2)}$ to the exact
sequences (v) or (vi) in Proposition~\ref{prop:exact sequences}, we
obtain an exact sequence: \eqn &&\hs{-3ex}  0 \rightarrow
V(\pi^{(2)}(\ell_1,(-q)^{a'+b})) \otimes
V(\pi^{(2)}(\ell_2,(-q)^{a+b'}))
  \rightarrow V(\pi^{(2)}(\ell,x))
\otimes V(\pi^{(2)}(\ell',x')) \\
&&\hs{-3ex} \To[{\;\F^{(2)}(r)\;}]
  V(\pi^{(2)}(\ell',x')) \otimes V(\pi^{(2)}(\ell,x)) \rightarrow
  V(\pi^{(2)}(\ell_1,(-q)^{a'+b})) \otimes V(\pi^{(2)}(\ell_2,(-q)^{a+b'})) \rightarrow 0,
\eneqn where  $\ell_1 = b-a' +1$ and $\ell_2 = b'-a +1$.

Since $(\ell_1,(-q)^{a'+b})$ and $(\ell_2,(-q)^{a+b'})$ are not
linked in $\mathscr S_0(A^{(1)}_{N-1})$,
$\pi^{(2)}(\ell_1,(-q)^{a'+b})$ and $\pi^{(2)}(\ell_2,(-q)^{a+b'})$
are not linked in $\mathscr S_0(A^{(2)}_{N-1})$, either. It follows
that
$$V(\pi^{(2)}(\ell_1,(-q)^{a'+b})) \otimes V(\pi^{(2)}(\ell_2,(-q)^{a+b'}))
  \ncong V(\pi^{(2)}(\ell,x)) \otimes V(\pi^{(2)}(\ell',x')),$$
because the left hand side is simple but the right hand side is not.
Therefore $\F^{(2)}(r)$ does not vanish, as desired.
\end{proof}

\begin{theorem} \label{thm:irreducible to irreducible}
Let $M$ be a finite-dimensional simple  graded  $R(\ell)$-module and
$$\big ((a_1,b_1), \ldots , (a_r,b_r) \big)$$
be the multisegment associated with $M$. Set $\ell_k=b_k-a_k+1$.
\bnum
\item If $\ell_k > N$ for some $1 \leq k \leq r$, then $\F^{(2)}(M) \simeq 0$.
\item If $\ell_k \leq N$ for all $1 \leq k \leq r$, then $\F^{(2)}(M)$ is simple.
\ee
\end{theorem}
\begin{proof} (i) follows from Proposition \ref{prop:image of fund.
repns.}.

By Proposition \ref{prop:multisegments},  we have $M \cong \Im
\rmat{(a_1,b_1),  \ldots , (a_r,b_r)}$. On the other hand, if
$\ell_k \leq N$ for all $1 \leq k \leq r$, then by the above Lemma,
we know that $\F^{(2)}(r_{(a_1,b_1),  \ldots , (a_r,b_r)})$
 is a constant multiple of a composition of normalized $R$-matrices.
 It follows that
 $\Im \F^{(2)}(\rmat{(a_1,b_1),  \ldots , (a_r,b_r)})$ is simple.
 Hence we conclude that
 $$\F^{(2)}(M) \cong \F^{(2)}(\Im \rmat{(a_1,b_1),  \ldots , (a_r,b_r)}) \cong \Im \F^{(2)}(r_{(a_1,b_1),  \ldots , (a_r,b_r)}) $$ is simple, as desired.
\end{proof}

 \subsection{Quotient of  the category $R\gmod$}
 We will recall the quotient category of $R\gmod$ introduced in \cite[\S 4.4]{KKK13A}.
Set $\As_{\alpha} = R({\alpha}) \gmod$ and set $\As =
\soplus_{{\alpha} \in \rootl_J^+} \As_{\alpha}$. Similarly, we
define $\Ab_\alpha$ and $\Ab$ by $\Ab_\alpha=\Modg(R(\alpha))$ and
$\Ab=\soplus_{{\alpha} \in \rootl_J^+}\Ab_\alpha$. Then we have a
functor $\F^{(2)} =\soplus_{{\alpha} \in \rootl_J^+}
\F^{(2)}_\alpha\col \Ab \rightarrow \Mod(\UAtwo)$, where
$\F^{(2)}_{\alpha}$ is the functor  given in \eqref{eq:the functor}.

Let $\Ss_N$ be the smallest Serre subcategory of $\As$ (see
\cite[Appendix B.1]{KKK13A} ) such that \eq&&
\parbox{70ex}{
\begin{enumerate}
\item $\Ss_N$ contains $L(a, a+N)$ for any $a\in\Z$,
\item $X \conv Y, \ Y \conv X \in \Ss_N$
for all $X \in \As$ and $Y \in \Ss_N$.
\end{enumerate}}
\eneq Note that $\Ss_N$ contains $L(a,b)$ if $b\ge a+N$.

Let us denote by $\As/ \Ss_N$ the quotient category of $\As$ by
$\Ss_N$ and denote by $\mathcal Q\col \As \rightarrow \As/ \Ss_N$
the canonical functor. Since $\F^{(2)}$ sends $\Ss_N$ to $0$, the
functor $\F^{(2)} \col \As \rightarrow \UAtwo \smod$ factors through
$\mathcal Q$ by \cite[Theorem B.1.1 (v)]{KKK13A}: \eqn&&
\xymatrix@C=6ex@R=4ex{\As\ar[r]^{\mathcal{Q}}\ar[dr]_-{\F^{(2)}}
&\As/ \Ss_N\ar[d]^-{\F^{(2)'}}\\
&{\UAtwo\smod} } \eneqn
 That is, there exists a unique functor
$\F^{(2)'}\col \As/\Ss_N\to \UAtwo \smod$  up to an isomorphism such
that  the above diagram quasi-commutes.

\smallskip

Note that $\As$ and $\As/\Ss_N$ are tensor categories with the
convolution as tensor products. The module $R(0)\simeq \cor$ is a
unit object. Note also that $Q\seteq q R(0)$ is an invertible
central object of $\As/\Ss_N$ and $X\mapsto  Q\conv X\simeq X\conv
Q$ coincides with the grade shift functor.
 Moreover, the
functors $\mathcal Q$, $\F^{(2)}$ and $\F^{(2)'}$ are tensor
functors.

\smallskip

Similarly, we define $\Sb_N$ as the smallest Serre subcategory of
$\Ab$ such that \eq&&
\parbox{70ex}{
\begin{enumerate}
\item $\Sb_N$ contains $L(a, a+N)$,
\item $X \conv Y, \ Y \conv X \in \Sb_N$
for all $X \in \Ab$, $Y \in \Sb_N$,
\item $\Sb_N$ is stable under (not necessarily finite) direct sums.
\end{enumerate}}
\eneq Then we can easily see that $\Sb_N\cap \As=\Ss_N$ and hence
the functor $\As/ \Ss_N \to\Ab/\Sb_N$ is fully faithful.

The following proposition is proved  in \cite[\S 4.4]{KKK13A}, and
its corollary below can be proved similarly to \cite[Corollary
4.4.2]{KKK13A}.
\begin{prop} [{\cite[Proposition 4.4.1]{KKK13A}}]\hfill
\bnum
\item If an object $X$ is simple in $\As / \Ss_N$, then
there exists a  simple object  $M$ in $\As$ satisfying

\bna
\item $\mathcal Q(M) \simeq X$,
\item $b_k-a_k+1 \leq N$ for $1 \leq k \leq r$, where
$\big ( (a_1,b_1), \ldots , (a_r,b_r) \big)$ is the multisegment
associated with $M$. \ee
\item Let $\big ( (a_1,b_1), \ldots , (a_r,b_r) \big)$  be the multisegment associated with a simple  object  $M$ in $\As$.
If $b_k-a_k+1 \leq N$ for $1 \leq k \leq r$, then $\mathcal Q(M)$ is
simple in $\As / \Ss_N$. \ee
\end{prop}

\begin{corollary}
The functor $\F^{(2)'}\col\As/\Ss_N\to \UAtwo\smod$ sends simple
objects in $\As/\Ss_N$ to simple objects in $\UAtwo\smod$.
\end{corollary}

\bigskip
\subsection{The categories $\T'_N$ and $\T_N$} \label{subsec: T_N}
In this section, we recall the categories $\T'_N$ and $\T_N$
introduced (and  denoted by  $T'_J$ and $T_J$, respectively) in
\cite[\S4.5]{KKK13A}.

\begin{definition}
Let $S$ be the automorphism of $P_J\seteq\soplus_{a\in\Z}\Z\oep_a$
given by $S(\oep_a)=\oep_{a+N}$. We define the bilinear form $B$ on
$P_J$ by \eq &&B(x,y)=-\sum_{k>0}(S^kx,y)\qtext{for $x,y\in P_J$.}
\label{def:S} \eneq
\end{definition}

\begin{definition} We define the new tensor product $\nconv\col\Ab/\Sb_N\times
\Ab/\Sb_N\to\Ab/\Sb_N$ by \eqn X\nconv Y=q^{B(\al,\beta)} X\conv
Y\simeq Q^{\tens B(\al,\beta)}\conv X\conv Y, \eneqn
 where  $X\in(\Ab/\Sb_N)_{\al},$ $Y\in(\Ab/\Sb_N)_{\beta}$ and $Q=q\,\one$.
\end{definition} Then $\Ab/\Sb_N$ as well as $\As/\Ss_N$ is endowed with a new
structure of tensor category by $\nconv$ as shown in \cite[Appendix
A.8]{KKK13A}.

Set
\begin{equation}
L_a\seteq L(a, a+N-1)  \quad \text{and} \quad
  u_a\seteq u(a,a+N-1)\in L_a
\quad\text{for} \  a \in \Z.
\end{equation}

For $a,j\in\Z$, set \eq \label{eq:faj} &&f_{a,j}(z) \seteq
(-1)^{\delta_{j,a+N}}z^{-\delta(a\le j<a+N-1)-\delta_{j,a+N}} \in
\cor[z^{\pm1}]. \eneq

\begin{theorem}[{\cite[Theorem 4.5.8]{KKK13A}}] \label{thm:L_a commutes with X}
The following statements  hold. \bnum
\item
$L_a$ is a central object   in  $\As/\Ss_N$ \ro see \cite[Appendix
A.3]{KKK13A} \rf; i.e., \bna
\item
$f_{a,j}(z)R_{L_a,L(j)_z}$ induces an isomorphism in $\As/\Ss_N$
$$R_a(X)\col L_a\nconv X\isoto X\nconv L_a$$
 functorial in  $X\in\As/\Ss_N$,
\item the diagram
$$\xymatrix@C=10ex
{L_a\nconv X\nconv Y\ar[r]^{R_a(X)\nconv Y}\ar[dr]_-{R_a(X\nconv
Y)\hs{2ex} }&
X\nconv L_a\nconv Y\ar[d]^{X\nconv R_a(Y)}\\
& X\nconv Y \nconv L_a }
$$ is commutative in $\As/\Ss_N$ for any $X,Y\in\As/\Ss_N$.
\ee
\item The  isomorphism
$R_a(L_a)\col L_a\nconv L_a\isoto L_a\nconv L_a$ coincides with
$\id_{L_a\nconv L_a}$ in $\As/\Ss_N$.
\item
For $a,b\in\Z$,
 the  isomorphisms
\eqn
 R_a(L_b)\col L_a\nconv L_b\isoto L_b\nconv L_a \ \text{and} \ R_b(L_a)\col L_b\nconv L_a\isoto
L_a\nconv L_b \eneqn

 in $\As/\Ss_N$ are  the inverses  to each other.
\ee \end{theorem}

\medskip

By the preceding theorem, $\{(L_a, R_a)\}_{a \in J}$ forms a
commuting family of central objects in $(\AA / \Ss_N , \nconv)$ (
\cite[Appendix A. 4]{KKK13A}). Following \cite[Appendix A.
6]{KKK13A}, we localize $(\AA / \Ss_N , \nconv)$ by this commuting
family. Let us denote by  $\T'_N$ the resulting category $(\AA /
\Ss_N)[L_a^{\nconv -1}\mid a\in J]$ . Let $\Upsilon \colon \AA /
\Ss_N \to \T'_N$ be the  projection functor. We denote by $\T_N$ the
tensor category $(\As / \Ss_N)[L_a\simeq\one\mid a\in J]$ and by
$\Xi \col \T'_N \to \T_N$ the canonical functor (see \cite[Appendix
A.7]{KKK13A} and \cite[Remark 4.5.9]{KKK13A}). Thus we have a chain
of tensor functors
$$\As \To[\ {\mathcal Q}\ ]\As/\Ss_N\To[\ \Upsilon\ ] \T'_N\seteq(\AA / \Ss_N)[L_a^{\nconv -1}\mid a\in J]
\To[\ \Xi\ ] \T_N\seteq(\As / \Ss_N)[L_a\simeq\one\mid a\in J].$$

The categories $\T_N$ and $\T'_N$ are rigid tensor categories; i.e.,
every object has a right dual and a left dual  (\cite[Theorem
4.6.3]{KKK13A}).

\medskip

In the rest of this section, we will show that the functor
$\F^{(2)'}$ factors through the category  $\T_N$. First, we need

\begin{lemma} For $b\in J$, set $V_k=V^{(2)}(\varpi_1)_{q^{-2k}}$  $(1\le
k\le N)$, $W=V_{N}\tens V_{N-1}\tens\cdots \tens V_{1}$, and choose
an epimorphism $\vphi\col W\to \cor$ in $\UAtwo\smod$.
 Let
$$\Rnorm_{W,V_z}\col W\tens V_z\to V_z\tens W$$
be the $R$-matrix obtained by the composition of normalized
$R$-matrices
$$V_{N}\tens\cdots \tens V_{1}\tens V_z
\To[\Rnorm_{V_1,V_z}]V_{N}\tens\cdots \tens V_{2} \tens V_z\tens
V_1\To\cdots \To[\Rnorm_{V_N,V_z}]V_z\tens V_{N}\tens\cdots \tens
V_{1},$$ and  let  $g(z)= q^{N}
\dfrac{(z-q^{-2N})(z+q^{-N-2})}{(z-q^{-2})(z+q^{-N})}$.

Then we have a commutative diagram \eq \xymatrix{
W\tens V_z\ar[r]^{\Rnorm_{W,V_z}}\ar[d]_{\vphi\tens V_z}&V_z\tens W\ar[d]^{V_z\tens \vphi}\\
\cor\tens V_z\ar[r]^{g(z)}& V_z\tens \cor. }\label{diag:RV} \eneq
\end{lemma} \begin{proof} Let $a_{11}(z)$ be the function satisfying
$\Runiv_{V^{(2)}(\varpi_1), V^{(2)}(\varpi_1)} = a_{11}(z)
\Rnorm_{V^{(2)}(\varpi_1), V^{(2)}(\varpi_1)}$, where
$\Runiv_{V^{(2)}(\varpi_1), V^{(2)}(\varpi_1)}$ denotes the
universal $R$-matrix
 between $V^{(2)}(\varpi_1)$ and $V^{(2)}(\varpi_1)$ (
 see \cite[Appendix A]{AK}).
Then the diagram \eqref{diag:RV} is commutative, if
$g(z)=\displaystyle \prod_{k=1}^N a_{11}(q^{2k} z)^{-1}$.

By \cite[(4.11)]{Oh14}, we have
\begin{align}
 a_{11}(z) &= q \dfrac{[ N+2 ]' [ N-2 ]'}{[ N ]' [ N ]'}
                    \dfrac{[0][2N]}{[2][2N-2]} , \label{eq:a11}
\end{align}
where $(z;q)_\infty= \prod_{s=0}^\infty (1-q^sz)$, $[a]=((-q)^{a}z ;
p^{*2})_\infty$, $[ a ]'=(-q^{a}z ; p^{*2})_\infty$ and
$p^*=-q^{N}$. It follows that \eqn \displaystyle\prod_{k=1}^N
a_{11}(q^{2k} z) &&=q^N \displaystyle\prod_{k=1}^N  \dfrac{[ N+2+2k
]' [ N-2+2k ]'}{[ N+2k ]' [ N+2k ]'}
                    \dfrac{[2k][2N+2k]}{[2+2k][2N-2+2k]} \\
&&= q^N  \dfrac{[ 3N+2 ]' [ N ]'}{[ N+2 ]' [ 3N ]'}
                    \dfrac{[2][4N]}{[2N+2][2N]}
                    =q^{-N} \dfrac{(z+q^{-N})(z-q^{-2})}{(z+q^{-N-2})(z-q^{-2N})} \\
&&=g(z)^{-1},
 \eneqn
as desired. \end{proof} The proof of the following lemma is
straightforward. \begin{lemma} \label{lem:cij} Let
$\set{\psi_k(z)}{k\in\Z} \subset \cor[[z]]^\times$ be a family of
power series such that \eq \label{cond:psi_a} \psi_a(0)
\psi_{-a-N+1}(0)=1 \ ( a \in \Z), \quad \displaystyle
\prod_{k=1-N}^{0}\psi_k(0)=1. \eneq Set \eqn
&&\phi_0(z) \seteq1,\\
&&\phi_k(z) \seteq \psi_{k-N}(z)^{-1} \quad (1 \le k \le N-2),\\
&&\phi_{N-1}(z) \seteq \psi_0(z) \displaystyle \prod_{k=1}^{N-2} \phi_k(z)^{-1}, \\
&&\phi_k(z) \seteq
\begin{cases}
\dfrac{\psi_{k-N+1}(z)}{\psi_{k-N}(z)} \phi_{k-N}(z) & (k \ge N), \\
\dfrac{\psi_{k}(z)}{\psi_{k+1}(z)} \phi_{k+N}(z) & (k \le -1).
\end{cases}
\eneqn Then $\set{\phi_k(z)}{k\in\Z}$ satisfies \eqn
&&\phi_a(0) \phi_{-a}(0) =1 \quad (a \in\Z), \\
&&\displaystyle \prod_{k=a}^{a+N-1}\phi_k(z)=\psi_a(z) \quad (a \in
\Z).\eneqn \end{lemma}

Now we make a special choice of $c_{ij}$ in $\S$ \ref{sec:SWfunctor}
as follows: For $a \in \Z$, set \eqn \psi_a(z) &&\seteq
(-1)^{\delta(1-N \le a \le -1\ )}
 (z^{\delta(a=0)-\delta(a=1-N)})
\ g(q^{2(-a-N)}(z+1))^{-1} \\
&&=(-1)^{\delta(1-N\le a\le -1)}
\dfrac{(z+1+q^{2a+N})(z+1-q^{2a+2N-2})^{\delta(a\neq 1-N)}}
{(z+1+q^{2a+N-2})(z+1-q^{2a})^{\delta(a\neq 0)}} \in
\cor[[z]]^\times. \eneqn Then it is straightforward to check
$\{\psi_a(z)\}_{a \in \Z}$ satisfy the conditions
\eqref{cond:psi_a}.
 Define $\{\phi_a(z)\}_{a\in\Z}$ as in Lemma \ref{lem:cij}.
 Finally, set
 \eq
 c_{i,j}(u,v) \seteq \dfrac{\phi_{i-j}(v)}{\phi_{j-i}(u)} \phi_{j-i}(0).
 \label{eq:c}
 \eneq
Note that $\{c_{i,j}(u,v)\}_{i,j \in \Z}$ satisfy the  condition
\eqref{cond:cij} by the construction.

\begin{theorem} \label{thm:cij}
 If we choose  $\{c_{i,j}(u,v)\}_{i,j \in \Z}$ as above,
 then the diagram \cite[(A.7.1)]{KKK13A}  is commutative  for  the  functor $\F^{(2)'}\col\As/\Ss_N\to \UAtwo\smod$
 and  the  commuting family of central  objects  $\{(L_a,R_a)\}_{a\in J}$. That is,
 the diagram
\eqn &&\hs{-4ex} \ba{l}\xymatrix@C=8.5ex{ \F^{(2)'}(L_a\nconv
M)\ar[d]^-{\F^{(2)'}(R_a(M))}\ar[r]^-\sim
&\F^{(2)'}(L_a)\tens\F^{(2)'}(M)\ar[r]^-{ g_a \tens \F^{(2)'}(M)}
&\cor\tens \F^{(2)'}(M)\ar[dr]\\
\F^{(2)'}(M\nconv
L_a)\ar[r]^-\sim&\F^{(2)'}(M)\tens\F^{(2)'}(L_a)\ar[r]^-{\F^{(2)'}(M)\tens
g_a} &\F^{(2)'}(M)\tens  \cor  \ar[r]&\F^{(2)'}(M)
}\ea\label{eq:1com} \eneqn
 is commutative  for any isomorphism $ g_a \col\F'(L_a)\isoto\cor$.
\end{theorem}

\begin{proof} By the same argument in \cite[Theorem 4.6.5]{KKK13A}, it is
enough to show that \eq&&f_{a,j}(z) \ g(q^{2(j-a-N)}(z+1))
\prod_{a\le k\le a+N-1}P_{k,j}(0,z)=1
 \label{eq:fajPkj}\eneq
 for all $a,j \in \Z$.
Recall that  $P_{k,j}(0,z)=c_{k,j}(0,z)(-z)^{\delta(j=k+1)}$ and
\eqn f_{a,j}(z)= (-1)^{\delta_{j,a+N}} z^{-\delta(a \le j <
a+N-1)-\delta_{j,a+N}}. \eneqn Hence it amounts to showing that \eqn
\label{eq:ck0z}&&\hs{-5ex}\ba{rcl} \displaystyle\prod_{k=a}^{a+N-1}
c_{k,j}(0,z)&=& (-1)^{\delta(a+1\le j\le a+N-1)}
 (z^{\delta(j=a)-\delta(j=a+N-1)})
\ g(q^{2(j-a-N)}(z+1))^{-1} \ea \eneqn for all $a,j \in \Z$. Since
$c_{i+1,j+1}(u,v) = c_{i,j}(u,v)$ for all $i,j \in\Z$, we have only
to show that

\eqn \displaystyle\prod_{k=a}^{a+N-1} c_{k,0}(0,z)&&=
(-1)^{\delta(a+1\le 0\le a+N-1)}
 (z^{\delta(a=0)-\delta(0=a+N-1)})
\ g(q^{2(-a-N)}(z+1))^{-1} \\
&&= \psi_a(z). \eneqn for all $a \in\Z$.

Since $c_{k,0}(0,z)=\phi_k(z)$, we obtain the desired result.
\end{proof}

Hence, \cite[Proposition A.7.3]{KKK13A} implies that the functor
$\F^{(2)'} \colon \As/\Ss_N\to\UAtwo\smod$ factors through $\T_N$.
Consequently, we obtain a tensor functor $\tF^{(2)} \col \T_N\to
\UAtwo\smod$ such that the following diagram quasi-commutes:
\begin{equation}
 \xymatrix@C=10ex{
\As\ar[r]^-{\mathcal
Q}\ar[drr]_-{\F^{(2)}}&\As/\Ss_N\ar[r]^{\Upsilon
}\ar[rd]^(.55){\F^{(2)'}}
&\T'_N\ar[d]\ar[r]^{ \Xi  }&\T_N\ar[dl]^-{\tF^{(2)}}\\
&&\UAtwo\smod\,.}
\end{equation}

Moreover, by \cite[Proposition A.7.2]{KKK13A}, we obtain
\begin{prop}
  The functor $\tF^{(2)}$ is exact.
\end{prop}

Recall that  $\CC^0_\g$ is the smallest full subcategory of $\CC_\g$
stable under taking subquotients,  extensions, tensor products and
containing $\set{V(\varpi_i)_x}{(i,x) \in \mathscr S_0(\g)}$. By
Proposition \ref{prop:multisegments}  and Proposition
\ref{prop:image of fund. repns.}, the images of the functors
$\F^{(2)}$, $\F^{(2)'}$ and $\tF^{(2)}$ are inside the category
$\CC^0_\g$.

Let us denote by  $\Irr(\T_N)$ the set of the isomorphism classes of
simple objects in $\T_N$.
 Define an equivalence relation $\sim$ on $\Irr(\T_N)$ by $X \sim Y$ if and only if
$X \simeq q^c Y$ in $\T_N$ for some integer $c$. Let
$\Irr(\T_N)_{q=1}$ be a set of representatives of elements in
$\Irr(\T_N) / \sim $. Then the set $\Irr(\T_N)_{q=1}$  is isomorphic
to the set of ordered multisegments
$$\bl (a_1,b_1),\ldots (a_r,b_r) \br$$ satisfying
\begin{equation} \label{eq:segments<N}
  b_k -a_k+1 <N \ \text{for any} \
    1 \le k \le r.
\end{equation}

Since the proofs of the following proposition and theorem are
similar to the ones in \cite[\S4,7]{KKK13A}, we omit them.
\begin{prop} The functor $\tF^{(2)}\col \T_N \to
\CC^0_{A^{(2)}_{N-1}}$ induces a bijection between
$\Irr(\T_N)_{q=1}$ and $\Irr(\CC^0_{A^{(2)}_{N-1}})$, the set of
isomorphism classes of simple objects in $ \CC^0_{A^{(2)}_{N-1}}$.
\end{prop}

\begin{theorem} The functor $\tF^{(2)}\col\T_N\to \CC^0_{A^{(2)}_{N-1}}$ induces
a ring  isomorphism
\begin{equation*}
\phi_{\tF^{(2)}} \col K(\T_N)/(q-1)K(\T_N) \isoto
K(\CC^0_{A^{(2)}_{N-1}}).
\end{equation*}
\end{theorem}

Recall that in \cite[\S 4.6]{KKK13A}, we obtained a functor
$\tF^{(1)} : \T_N \To \CC^0_{A^{(1)}_{N-1}} $, where $\tF^{(1)}$,
$\T_N$ and $\CC^0_{A^{(1)}_{N-1}}$ were denoted by $\tF$, $\T_J$ and
$\mathcal C_J$, respectively. The functor $\tF^{(1)}$ also induces a
ring isomorphism
\begin{equation*}
\phi_{\tF^{(1)}} \col K(\T_N)/(q-1)K(\T_N) \isoto
K(\CC^0_{A^{(1)}_{N-1}}).
\end{equation*}

\begin{theorem}\label{th: dim} Let $M$ be a simple object in $\T_N$. Then we
have \eq \label{eq:dim} \dim_\cor \tF^{(1)}(M) = \dim_\cor
\tF^{(2)}(M). \eneq \end{theorem} \begin{proof} By Proposition
\ref{prop:dim} and Proposition \ref{prop:image of fund. repns.}, we
know that \eqref{eq:dim} holds when $M$ is a $1$-dimensional module
corresponding to a segment.

Note that the assignment
$$(a,b) \mapsto \alpha_a + \alpha_{a+1} +\cdots + \alpha_{b}$$
gives a bijection between the set of segments and the set of
positive roots  of type $A_\infty$. Under this bijection, the order
$>$ on the set of segments induces a convex order $>$ of the set of
positive roots; i.e., we have $\alpha < \alpha+ \beta < \beta$, if
$\alpha,\beta, \alpha+\beta$ are positive roots  and $\alpha <
\beta$.
 It is not difficult to see $\set{L(a,b)}{(a,b) \ \text{is a segment}}$ is the cuspidal system
 corresponding to the above convex order in the sense of \cite[Definition 3.2]{Kle14}.
For $\gamma \in \rootl^+_J$, we denote by ${\rm KP}(\gamma)$ the set
of ordered multisegments
 such that the sum of the corresponding roots is  equal to $\gamma$.

Let $\bl(a_1,b_1), (a_2,b_2), \ldots (a_r,b_r)\br$ be the ordered
multisegment associated with a simple object $M$ in $\T_N$. Then, by
\cite[Theorem 3.1]{McNa12} (cf. \cite[Theorem 3.15 (iv)]{Kle14}),
 every composition factor of ${\rm rad}\bl  L(a_1,b_1) \conv
\cdots \conv L(a_r,b_r)\br$  has an associated multisegment
$$\bl (a'_1,b'_1), \ldots, (a'_s, b'_s) \br,$$
satisfying $(\al_{a'_1} + \cdots+ \al_{b'_1}) + \cdots + (\al_{a'_s}
+ \cdots+ \al_{b'_s})  \in \rm KP(\gamma)$,
 where $\gamma= (\al_{a_1} + \cdots+ \al_{b_1}) + \cdots + (\al_{a_r} + \cdots+ \al_{b_r})$,
and $\bl (a'_1,b'_1), \ldots, (a'_s, b'_s) \br  \prec \bl (a_1,b_1),
\ldots, (a_r, b_r) \br$. Here  $\prec$ denotes the bi-lexicographic
partial order on $\rm KP(\gamma)$ given in \cite[\S 3]{McNa12} (cf.
\cite[\S 3.2]{Kle14}). In particular, if $\bl(a_1,b_1),  \ldots
(a_r,b_r)\br$ is a minimal element in ${\rm KP}(\gamma)$ with
respect to $\prec$, then $M \simeq L(a_1,b_1) \conv \cdots \conv
L(a_r,b_r)$. Thus we obtain\eqref{eq:dim} in this case.

Now by induction on $\prec$, we may assume that
 \eqn
\dim_\cor \tF^{(1)}\bl{\rm rad} \bl L(a_1,b_1) \conv \cdots \conv
L(a_r,b_r) \br \br = \dim_\cor \tF^{(2)}\bl {\rm rad} \bl L(a_1,b_1)
\conv \cdots \conv L(a_r,b_r) \br\br.\eneqn It follows that \eqn
\dim_\cor \tF^{(1)}(M) &&=\displaystyle \prod_{k=1}^{r} \dim_\cor
\tF^{(1)}(L(a_k,b_k))
- \dim_\cor \tF^{(1)} \bl {\rm rad} \bl L(a_1,b_1) \conv \cdots \conv L(a_r,b_r)\br\br \\
&&=\displaystyle \prod_{k=1}^{r} \dim_\cor \tF^{(2)}(L(a_k,b_k))
 - \dim_\cor \tF^{(2)}\bl{\rm rad} \bl L(a_1,b_1) \conv \cdots \conv L(a_r,b_r)\br\br \\
&&=\dim_\cor \tF^{(2)}(M), \eneqn as desired. \end{proof}

Set
$$\phi^{(2)} \seteq \phi_{\tF^{(2)}} \circ \phi_{\tF^{(1)}}^{-1} : K(\CC^0_{A^{(1)}_{N-1}}) \isoto K(\CC^0_{A^{(2)}_{N-1}}).$$

\begin{corollary} \label{cor: iso bwn gro ring} The ring isomorphism $\phi^{(2)}$
induces a bijection between $\Irr(\CC^0_{A^{(1)}_{N-1}})$ and
$\Irr(\CC^0_{A^{(2)}_{N-1}})$, which preserves the dimensions.
\end{corollary}

%\begin{remark}
%In \cite[\S 4.4]{Her101}, Hernandez showed that there is a ring
%isomorphism $\pi$ between the Grothendieck rings $K(
%\CC^0_{A^{(1)}_{N-1}})$ and $K(\CC^0_{A^{(2)}_{N-1}})$, which sends
%Kirillov-Reshetikhin modules to
% Kirillov-Reshetikhin modules.
%But it was not known whether $\pi$ sends simple modules to
% simple modules.
%The above corollary is an affirmative answer to this question.
%\end{remark}

\end{document}